\documentclass[11pt]{article}
\usepackage[margin=1in]{geometry} 

\title{Bounds on the volume of an inclusion in a body from a complex conductivity measurement}
\author{Andrew E. Thaler and Graeme W. Milton \\ \small{Department of Mathematics, University of Utah, Salt Lake City, UT 84112, USA}}
\date{\empty} 

\usepackage{subcaption} 
\usepackage{caption}

\usepackage{graphicx}  
\usepackage{epstopdf} 
\usepackage{flafter}  

\usepackage{amsmath,amssymb,amsthm}  
\usepackage{bm}  

\usepackage[pdftex,bookmarks,colorlinks,breaklinks]{hyperref}  

\numberwithin{equation}{section}

\newcommand{\myspace}{\\[0.4cm]}

\newcommand{\ii}{\mathrm{i}}
\newcommand{\ee}{\mathrm{e}}
\newcommand{\xhat}{\widehat{\mathbf{x}}}
\newcommand{\yhat}{\widehat{\mathbf{y}}}
\newcommand{\x}{\mathbf{x}}
\newcommand{\lla}{\left\langle}
\newcommand{\rra}{\right\rangle}
\newcommand{\E}{\mathbf{E}}
\newcommand{\J}{\mathbf{J}}
\newcommand{\dets}{\sigma_1^{(1)}\sigma_2^{(2)} - \sigma_2^{(1)}\sigma_1^{(2)}}
\newcommand{\Rp}{R_{\perp}}

\newcommand{\ga}{\mathbf{g}^{(\alpha)}}
\newcommand{\ha}{\mathbf{h}^{(\alpha)}}
\newcommand{\ca}{\chi^{(\alpha)}}
\newcommand{\fa}{f^{(\alpha)}}
\newcommand{\Sa}{S^{(\alpha)}}

\newcommand{\Ma}{M^{(\alpha)}}
\newcommand{\Ta}{T^{(\alpha)}}
\newcommand{\Ba}{B^{(\alpha)}}

\newcommand{\pa}{p^{(\alpha)}}
\newcommand{\tpa}{\widetilde{p}^{(\alpha)}}
\newcommand{\al}{a^{(\alpha)}}
\newcommand{\aone}{a^{(1)}}
\newcommand{\atwo}{a^{(2)}}

\newcommand{\Eoo}{\E_1^{(1)}}
\newcommand{\Eot}{\E_1^{(2)}}
\newcommand{\Eto}{\E_2^{(1)}}
\newcommand{\Ett}{\E_2^{(2)}}

\newcommand{\A}{\mathcal{A}}
\newcommand{\F}{\mathcal{F}}
\newcommand{\Fa}{\mathcal{F}^{(\alpha)}}
\newcommand{\Fo}{\mathcal{F}^{(1)}}
\newcommand{\Ft}{\mathcal{F}^{(2)}}
\newcommand{\Ea}{\mathcal{E}^{(\alpha)}}

\newcommand{\Ei}{\mathcal{E}^{(1)}_f\cap\mathcal{E}^{(2)}_f}
\newcommand{\cE}{\mathcal{E}}
\newcommand{\tF}{\widetilde{\mathcal{F}}}
\newcommand{\tFa}{\widetilde{\mathcal{F}}^{(\alpha)}}
\newcommand{\tFo}{\widetilde{\mathcal{F}}^{(1)}}
\newcommand{\tFt}{\widetilde{\mathcal{F}}^{(2)}}
\newcommand{\tA}{\widetilde{\mathcal{A}}}
\newcommand{\tEa}{\widetilde{\mathcal{E}}^{(\alpha)}}
\newcommand{\tEo}{\widetilde{\mathcal{E}}^{(1)}}
\newcommand{\tEt}{\widetilde{\mathcal{E}}^{(2)}}
\newcommand{\tcE}{\widetilde{\mathcal{E}}}

\newcommand{\xione}{\xi^{(1)}}
\newcommand{\xitwo}{\xi^{(2)}}
\newcommand{\etaone}{\eta^{(1)}}
\newcommand{\etatwo}{\eta^{(2)}}
\newcommand{\fone}{f^{(1)}}
\newcommand{\ftwo}{f^{(2)}}
\newcommand{\chione}{\chi^{(1)}}
\newcommand{\chitwo}{\chi^{(2)}}
\newcommand{\psione}{\psi^{(1)}}
\newcommand{\psitwo}{\psi^{(2)}}

\newcommand{\avgnormvplusminusa}{\langle\|\mathbf{v}_{\pm}^{(\alpha)}\|^2\rangle}
\newcommand{\avgnormvplusa}{\langle\|\mathbf{v}_{+}^{(\alpha)}\|^2\rangle}
\newcommand{\avgnormvminusa}{\langle\|\mathbf{v}_{-}^{(\alpha)}\|^2\rangle}

\newcommand{\avgnormvplusone}{\langle\|\mathbf{v}_{+}^{(1)}\|^2\rangle}
\newcommand{\avgnormvminusone}{\langle\|\mathbf{v}_{-}^{(1)}\|^2\rangle}

\newcommand{\avgnormvplustwo}{\langle\|\mathbf{v}_{+}^{(2)}\|^2\rangle}
\newcommand{\avgnormvminustwo}{\langle\|\mathbf{v}_{-}^{(2)}\|^2\rangle}

\newcommand{\normavgvplusminusa}{\|\langle\mathbf{v}_{\pm}^{(\alpha)}\rangle\|^2}
\newcommand{\normavgvplusa}{\|\langle\mathbf{v}_{+}^{(\alpha)}\rangle\|^2}
\newcommand{\normavgvminusa}{\|\langle\mathbf{v}_{-}^{(\alpha)}\rangle\|^2}

\newcommand{\normavgvplusone}{\|\langle\mathbf{v}_{+}^{(1)}\rangle\|^2}
\newcommand{\normavgvminusone}{\|\langle\mathbf{v}_{-}^{(1)}\rangle\|^2}

\newcommand{\normavgvplustwo}{\|\langle\mathbf{v}_{+}^{(2)}\rangle\|^2}
\newcommand{\normavgvminustwo}{\|\langle\mathbf{v}_{-}^{(2)}\rangle\|^2}

\newcommand{\vplusminusa}{\mathbf{v}_{\pm}^{(\alpha)}}
\newcommand{\vplusa}{\mathbf{v}_{+}^{(\alpha)}}
\newcommand{\vminusa}{\mathbf{v}_{-}^{(\alpha)}}

\newcommand{\vplusminusone}{\mathbf{v}_{\pm}^{(1)}}
\newcommand{\vplusone}{\mathbf{v}_{+}^{(1)}}
\newcommand{\vminusone}{\mathbf{v}_{-}^{(1)}}

\newcommand{\vplusminustwo}{\mathbf{v}_{\pm}^{(2)}}
\newcommand{\vplustwo}{\mathbf{v}_{+}^{(2)}}
\newcommand{\vminustwo}{\mathbf{v}_{-}^{(2)}}

\DeclareMathOperator{\tr}{tr}
\DeclareMathOperator{\Arg}{Arg}

\theoremstyle{definition}\newtheorem{feasibleregion}{Definition}[section]
\theoremstyle{definition}\newtheorem{elementaryfeasibleregion}[feasibleregion]{Definition}
\theoremstyle{definition}\newtheorem{ellipseregion}[feasibleregion]{Definiton}
\theoremstyle{remark}\newtheorem*{betaremark}{Remark}
\theoremstyle{remark}\newtheorem*{betaremark2}{Remark}
\theoremstyle{theorem}\newtheorem{elementary_bounds_theorem}{Theorem}[section]

\theoremstyle{theorem}\newtheorem{ellipseslemma}{Lemma}[section]

\theoremstyle{theorem}\newtheorem{boundingbox}[ellipseslemma]{Lemma}
\theoremstyle{remark}\newtheorem*{boundingboxremark}{Remark}
\theoremstyle{theorem}\newtheorem{twointersectionpoints}[ellipseslemma]{Lemma}
\theoremstyle{theorem}\newtheorem{noRperptheorem}{Theorem}[section]

\theoremstyle{definition}\newtheorem{feasibleregion2}{Definition}[section]
\theoremstyle{definition}\newtheorem{ellipseregion2}[feasibleregion2]{Definition}
\theoremstyle{definition}\newtheorem{restricted_elementary_admissible}[feasibleregion2]{Definition}
\theoremstyle{theorem}\newtheorem{eigenvaluelemma}{Lemma}[section]
\theoremstyle{theorem}\newtheorem{improved_elementary_bounds_theorem}{Theorem}[section]
\theoremstyle{theorem}\newtheorem{ellipseslemma2}[eigenvaluelemma]{Lemma}
\theoremstyle{theorem}\newtheorem{boundingbox2}[eigenvaluelemma]{Lemma}
\theoremstyle{theorem}\newtheorem{twointersectionpoints2}[eigenvaluelemma]{Lemma}
\theoremstyle{theorem}\newtheorem{Rperptheorem}[improved_elementary_bounds_theorem]{Theorem}

\begin{document}

\maketitle

\begin{abstract}
  We derive bounds on the volume of an inclusion in a body in two or three dimensions when the conductivities of the inclusion and the surrounding body are complex and assumed to be known.  The bounds are derived in terms of average values of the electric field, current, and certain products of the electric field and current.  All of these average values are computed from a single electrical impedance tomography measurement of the voltage and current on the boundary of the body.  Additionally, the bounds are tight in the sense that at least one of the bounds gives the exact volume of the inclusion for certain geometries and boundary conditions.
\end{abstract}

\section{Introduction}

Electrical impedance tomography (EIT) is a non-invasive imaging technique in which one utilizes measurements of the voltage and current at the boundary of a body $\Omega$ to determine information about the electrical properties inside $\Omega$.  EIT has applications in the non-destructive testing of materials, geophysical prospection, and medical imaging--see \cite{Borcea:2002:EIT, Cheney:1999:EIT} and references therein.  In the context of medical imaging, EIT can be used for breast cancer detection \cite{Cheney:1999:EIT} and the screening of organs for degradation prior to transplantation surgery \cite{Beretta:2011:SEE, Griffiths:1995:TSE}.  In these applications the complex conductivities of the healthy and cancerous/degraded tissues differ, so information about the conductivity distribution would allow one to estimate the location and/or size of the cancerous/degraded tissue.  For many other medical applications see \cite{Hamilton:2013:DEI} and the references therein.

Our goal in this paper is to find bounds on the volume fraction occupied by an inclusion $D$ inside a body $\Omega$.  In the context of organ screening, for example, $D$ could represent the degraded tissue and $\Omega \setminus D$ could represent the healthy tissue; it would be useful to estimate the volume of degraded tissue (the volume of $D$) before the organ is transplanted \cite{Beretta:2011:SEE, Griffiths:1995:TSE}.  We will assume that the complex conductivity inside $\Omega$ is of the form
\begin{equation*}
  \sigma = \sigma^{(1)} \chi(D) + \sigma^{(2)} \chi(\Omega \setminus D)
\end{equation*}
where $\sigma^{(\alpha)} = \sigma_1^{(\alpha)} + \ii \sigma_2^{(\alpha)}$ for $\alpha = 1, 2$ and $\chi(D)$ is the indicator function of $D$.   We require $\sigma_1^{(\alpha)} > 0$ for $\alpha = 1, 2$, which corresponds to energy dissipation \cite{Borcea:2002:EIT}.  More generally, we will follow \cite{Kang:2012:SBV} and consider a two-phase material with conductivity 
\begin{equation*}
  \sigma(\x) = \sigma^{(1)} \chione(\x) + \sigma^{(2)} \chitwo(\x)
\end{equation*}
where $\sigma^{(1)}$ and $\sigma^{(2)}$ are as before and $\chione$ is the characteristic function of phase 1, namely
\begin{equation*}
  \chione(\x) = 1-\chitwo(\x) = \begin{cases} 1 &\text{if } \x \in \text{phase } 1\\
  							0 &\text{if } \x \in \text{phase } 2.
			  \end{cases}
\end{equation*}
We will also assume that each phase is homogeneous and isotropic, so $\sigma^{(1)}$ and $\sigma^{(2)}$ are constant complex scalars (as discussed in \cite{Beretta:2011:SEE}, this is a reasonable assumption in the contexts of breast cancer detection and organ screening).

EIT operates in the quasistatic regime, where the wavelengths of all relevant electric and magnetic fields are much larger than $\Omega$.  In EIT, one typically prescribes either the voltage or current on $\partial \Omega$.  Under these conditions the voltage $V$ satisfies
\begin{equation}\label{conductivity_equation}
  \nabla \cdot \left(\sigma \nabla V\right) = 0 \text{ in } \Omega
\end{equation}
subject to either the Dirichlet boundary condition 
\begin{equation}\label{Dirichletbc}
  V = V_0 \text{ on } \partial \Omega
\end{equation}
or the Neumann boundary condition
\begin{equation}\label{Neumannbc}
  \left\{\begin{aligned} \sigma\frac{\partial V}{\partial n} &= I_0 \text{ on } \partial\Omega \\
  					   \int_{\partial\Omega} I_0 &= \int_{\partial \Omega} V = 0, \end{aligned}\right.
\end{equation}
where $\mathbf{n}$ is the outward unit normal to $\partial \Omega$ and $\frac{\partial V}{\partial n} = \nabla V \cdot \mathbf{n}$--see \cite{Borcea:2002:EIT}.  The PDE \eqref{conductivity_equation} can be equivalently written in the form
\begin{equation}\label{quasistatic_equations}
  \E = -\nabla V, \quad \nabla \cdot \J = 0, \quad \text{and} \quad \J = \sigma\E,
\end{equation}
where $\E$ is the electric field and $\J$ is the current density--see \cite{Borcea:2002:EIT}.  For a derivation of \eqref{conductivity_equation}, \eqref{quasistatic_equations} and the boundary conditions \eqref{Dirichletbc}, \eqref{Neumannbc} see \cite{Cheney:1999:EIT, Hamilton:2012:DDB}.

Our data will be the measurements $\left(V_0,\left.\sigma\frac{\partial V}{\partial n}\right|_{\partial \Omega}\right)$ when the Dirichlet boundary condition \eqref{Dirichletbc} is prescribed or $\left(I_0,\left.V\right|_{\partial \Omega}\right)$ when the Neumann boundary condition \eqref{Neumannbc} is prescribed.  (The measurements $\left.\sigma\frac{\partial V}{\partial n}\right|_{\partial \Omega}$ and $\left.V\right|_{\partial \Omega}$ are known as the Dirichlet-to-Neumann and Neumann-to-Dirichlet maps, respectively--see \cite{Borcea:2002:EIT} and the references therein for a more complete description and properties of these maps.  Also note that we are assuming that we know the voltage and current around the entire boundary $\partial \Omega$--see \cite{Borcea:2002:EIT, Hamilton:2013:NIP}).  Our goal is to use a single measurement of the voltage and current on $\partial \Omega$ to derive lower and upper bounds on the volume fraction of phase 1, namely $\fone = \langle\chione\rangle$, where
\begin{equation}\label{average_definition}
  \langle\mathbf{u}\rangle = \frac{1}{|\Omega|}\int_{\Omega} \mathbf{u}
\end{equation}
denotes the average of a vector-valed (or scalar) function $\mathbf{u}$ over $\Omega$ and $|\Omega|$ denotes the Lebesgue measure of $\Omega$.  

Several methods for deriving these bounds have been explored in the literature.  In the real conductivity case,  Alessandrini, Rosset, and Seo \cite{Alessandrini:2000:OSE}, Alessandrini and Rosset \cite{Alessandrini:1998:ICP}, Ikehata \cite{Ikehata:1998:SEI}, and Kang, Seo, and Sheen \cite{Kang:1997:ICP} utilized a single boundary measurement and methods from elliptic PDE to bound the volume of an inclusion $D$ in $\Omega$. In \cite{Alessandrini:1998:ICP, Alessandrini:2000:OSE} the authors made the technical assumption that 
\begin{equation}\label{assumption_on_D}
  d(D,\partial \Omega) \ge d_0 > 0
\end{equation}
where $d(D,\partial \Omega)$ is the distance between $D$ and $\partial\Omega$.
The bounds they derived involve constants that are not easy to determine.  Beretta, Francini, and Vessella \cite{Beretta:2011:SEE} used similar methods to derive bounds in the complex conductivity case--however they were able to remove the assumption \eqref{assumption_on_D} with certain restrictions on $\sigma^{(1)}$ and $\sigma^{(2)}$, which, as pointed out in their paper, is important in the application to organ screening as some of the degraded tissue may be present on the surface of the organ.  Their bounds also involve constants that in general may be difficult to determine, although they can be evaluated in some cases when special boundary conditions are imposed (see in particular Proposition 3.3 in their paper).

Capdeboscq and Vogelius \cite{Capdeboscq:2003:OAE} utilized multiple boundary measurements and the Lipton bounds on polarization tensors \cite{Lipton:1993:IEE} in the real conductivity case to find optimal asymptotic estimates on the volume of inclusions as the volume of the inclusions tends to 0.

If the body $\Omega$ contains a statistically homogeneous or periodic composite, then bounds on the effective tensors of this composite can be used in an inverse fashion to bound the volume fraction--see \cite{Cherkaeva:1998:IBM, McPhedran:1982:ESI, McPhedran:1990:ITP, Phan-Thien:1982:PUB}.  Similarly the universal bounds of Nemat-Nasser and Hori \cite{Nemat-Nasser:1993:MOP} on the response of a body $\Omega$ containing two phases in any configuration can be easily inverted to bound the volume fraction \cite{Milton:2012:UBE}.  Moreover Milton \cite{Milton:2012:UBE} used measurements of the voltage and current on $\partial \Omega$ with special boundary conditions to determine properties of the effective tensor of a composite containing rescaled copies of $\Omega$ packed to fill all space.  Bounds on this effective tensor led to universal bounds on the response of the body when the special boundary conditions are applied; these bounds were then inverted to bound the volume fraction.  We note that all of the bounds described in this paragraph can be computed in terms of known data (e.g. measurements of effective moduli or boundary measurements of the voltage and current).  

In the real conductivity case, variational methods have also been used to bound the volume fraction.  Several variational formulations of the PDE \eqref{conductivity_equation} were derived by Cherkaev and Gibiansky in \cite{Cherkaev:1994:VPC}.  Berryman and Kohn \cite{Berryman:1990:VCE} were the first to use variational methods in the context of EIT to determine information about the conductivity in a body.  Kang, Kim, and Milton \cite{Kang:2012:SBV} used the translation method introduced by Murat and Tartar \cite{Murat:1985:CVH, Tartar:1979:ECH, Tartar:1985:EFC} and independently by Lurie and Cherkaev \cite{Lurie:1982:AEC, Lurie:1984:EEC} (see also \cite{Milton:2002:TOC}) to derive sharp bounds on the volume fraction using 2 boundary measurements of the voltage and current in 2 dimensions.  The bounds are easily computed in terms of these measurements.  Kang, Kim, and Milton \cite{Kang:2012:SBV} also found geometries in which one of the bounds gives the true volume fraction.  Kang and Milton applied the translation method in 3 dimensions to find bounds on the volume fraction in \cite{Kang:2013:BVF3d}; these bounds can be computed using 3 boundary measurements.

Rather than derive variational principles, we will use the fact that certain variations are non-negative--see \eqref{g} and the paragraph following it, for example.  In \cite{Matheron:1993:QIP} Matheron used this idea to re-derive the famous Hashin-Shtrikman bounds \cite{Hashin:1962:VAT} on the effective conductivity of an isotropic composite--also see Section 16.5 of \cite{Milton:2002:TOC}.  We will also apply the ``splitting method'', introduced in in the context of elasticity in \cite{Milton:2012:BVF}, in which one derives bounds by splitting $\Omega$ into its constituent phases and correlating information about the fact that variations in each phase are non-negative and averages of certain quantities (null Lagrangians) are known.  Using this technique, in Theorems \ref{elementary_bounds_theorem} and \ref{improved_elementary_bounds_theorem} we establish some elementary bounds that can be computed from the single voltage and current measurement on $\partial \Omega$.  

In Theorems \ref{noRperptheorem} and \ref{Rperptheorem} we derive a method for numerically computing ``better'' bounds--we say ``better'' because these bounds may or may not be tighter than the above mentioned elementary bounds--see Section \ref{section_numerical_example}.  The method can be described as follows.  Let $f \in \A_e \subseteq (0,1)$, where $\A_e$ is an interval determined by the elementary bounds.  We call $f$ a \emph{test value}.  The splitting method implies that $f$ could potentially be the volume fraction of phase 1 if and only if certain $2\times 2$ matrices $S_f^{(1)}(x,y)$ and $S_f^{(2)}(x,y)$ (one for each phase) are simultaneously positive-semidefinite at some point in $(x,y) \in \mathbb{R}^2$.  This, in turn, is equivalent to requiring that two elliptic disks in the $(x,y)$ plane have a nonempty intersection.  (By elliptic disk we mean an ellipse in the plane union its interior).  In other words, if the elliptic disks do intersect, $f$ could be the true volume fraction; if the elliptic disks do not intersect, $f$ cannot be the true volume fraction.  This allows us to eliminate those values of $f \in \A_e$ for which the elliptic disks do not intersect, leaving us with a set $\A \subseteq \A_e$ of \emph{admissible} values.  Any $f\in\A$ could be the true volume fraction of phase 1, so bounds on $\A$ give us bounds on $f^{(1)}$.  Unfortunately these bounds must be computed numerically, but we emphasize that their computation is elementary and involves finding the interval (or intervals) of values where a certain function is positive and only requires a single measurement of the voltage and current on $\partial \Omega$.  

Finally, since we use the fact that variations are nonnegative rather than PDE methods or variational principles, we can easily determine attainability conditions for the bounds, i.e. conditions on the electric field that guarantee that the lower or upper elementary bound is exactly equal to the true volume fraction.  Our method also enables us to remove the assumption \eqref{assumption_on_D}; in fact, as long as the PDE \eqref{conductivity_equation} subject to the boundary conditions \eqref{Dirichletbc} or \eqref{Neumannbc} has a unique (weak) solution, our method can be applied.  Some of the bounds we obtain could presumably be obtained using the translation method, but the application of this method when we take into account all the null Lagrangians is less transparent since we would need to introduce a Lagrange multiplier for each of the many constraints.

The remainder of this paper is organized as follows.  In Section \ref{section_preliminaries} we introduce our notation and assumptions.  In Section \ref{section_the_splitting_method}, we apply the splitting method to several null Lagrangians, which are functionals of the electric field and current density that can be expressed in terms of the boundary voltage and current data.  In Section \ref{section_elementary_bounds} we derive the elementary bounds.  We derive a geometrical method for computing ``better'' bounds in Section \ref{section_more_sophisticated_bounds}.  Our work in Sections \ref{section_preliminaries}-\ref{section_more_sophisticated_bounds} applies in 2 or 3 dimensions.  In Section \ref{section_additional_null_Lagrangians} we use 2 additional null Lagrangians to derive even better bounds in the 2-D case, and in Section \ref{section_numerical_example} we apply our method to a test problem. 


\section{Preliminaries}\label{section_preliminaries}
As discussed in the introduction, we consider a two-phase mixture and also the case of an inclusion in a body.  
The region of interest (the unit cell of periodicity in the former case and the union of the inclusion and the body in the latter case) will be denoted by $\Omega$.  We assume that the conductivity in each phase is homogeneous and isotropic; then for $\x \in \Omega$ we have
\[\sigma(\x) = \sigma^{(1)} \chione(\x) + \sigma^{(2)} \chitwo(\x) \, ,\]
where $\sigma^{(\alpha)} = \sigma_1^{(\alpha)} + \ii \sigma_2^{(\alpha)}$ for $\alpha = 1,2$ are complex constants that we assume are known, $\sigma_1^{(\alpha)} > 0$ (as required physically), $0 < \left|\sigma^{(\alpha)}\right| < \infty$, and $\sigma^{(1)}\ne\sigma^{(2)}$.  We will see later that we must also assume 
\[\beta := \sigma_1^{(1)}\sigma_2^{(2)} - \sigma_2^{(1)}\sigma_1^{(2)} \ne 0,\]
so that $\Arg\sigma^{(1)} \ne \Arg\sigma^{(2)}$.  This implies that our results do not directly extend to the case when both phases have real conductivities.  

The average value of an integrable vector field (or scalar function) $\mathbf{u}$ is defined in \eqref{average_definition}.
The volume fraction of phase $\alpha$ is denoted by $\fa$, so
\[\fone = \langle\chione\rangle \quad \text{and} \quad \ftwo = 1-\fone = \langle\chitwo\rangle \, .\]
The electric potential, electric field, and current density will be denoted by $V = V_1 + \ii V_2, \E = \E_1 + \ii \E_2,$ and $\J = \J_1 + \ii \J_2$, respectively (so for $m = 1,2$, $V_m, \E_m,$ and $\J_m$ are real).  Recall that $V$ satisfies \eqref{conductivity_equation} subject to either \eqref{Dirichletbc} or \eqref{Neumannbc}, $\E = -\nabla V$, and $\J = \sigma \E$.   

Let $\mathbf{u} = \mathbf{u}_1 + \ii\mathbf{u}_2$ be a complex-valued vector field in $\mathbb{C}^2$ or $\mathbb{C}^3$.  Then we set $\mathbf{u}^{(\alpha)}(\x) := \ca(\x)\mathbf{u}(\x)$ and $\mathbf{u}_m^{(\alpha)}(\x) := \ca(\x)\mathbf{u}_m(\x)$ for $\alpha, m = 1, 2$.  The symbol ``$\cdot$'' will denote the usual Euclidean dot product on $\mathbb{R}^2$ or $\mathbb{R}^3$, while the Euclidean norm of a real-valued vector field $\mathbf{q}(\x) \in \mathbb{R}^2$ or $\mathbb{R}^3$ will be denoted by $\left\|\mathbf{q}(\x)\right\| = \sqrt{\mathbf{q}(\x)\cdot\mathbf{q}(\x)}$.  For any complex number $z = z_1 + \mathrm{i}z_2$ the modulus of $z$ will be denoted by $|z| = \sqrt{z_1^2 + z_2^2}$.


\section{The Splitting Method}\label{section_the_splitting_method}

\subsection{Null Lagrangians}
  We assume that we have full knowledge of a single applied boundary voltage $V_0$ and corresponding current $\sigma\frac{\partial V}{\partial n}|_{\partial \Omega}$ on $\partial \Omega$ (in the case of the Dirichlet problem--in the case of the Neumann problem, we assume that we have complete knowledge of the single applied current $I_0$ and corresponding voltage $V|_{\partial \Omega}$ on $\partial \Omega$--see the Introduction). In order to derive bounds on the volume fraction $\fone$ (hence $f^{(2)} = 1-f^{(1)}$) using this data, we make use of certain null Lagrangians, which are  functionals that can be expressed in terms of boundary data.  For $k, l = 1, 2$ we use integration by parts to find
\begin{equation}\label{nullLagrangians}
  \langle\E_k\rangle = -\dfrac{1}{\left|\Omega\right|} \displaystyle\int_{\partial \Omega} V_k \mathbf{n}; \quad \langle\J_l\rangle  = \dfrac{1}{\left|\Omega\right|} \displaystyle\int_{\partial \Omega} \mathbf{x} \left(\J_l \cdot \mathbf{n}\right); \quad \langle\E_k\cdot\J_l\rangle = -\dfrac{1}{\left|\Omega\right|}\displaystyle\int_{\partial \Omega} V_k \left(\J_l \cdot \mathbf{n}\right);
\end{equation}
$\mathbf{n}$ is the unit outward normal to $\partial \Omega$ and, in the 2-D case, all boundary integrals are taken in the positive (counterclockwise) orientation.
We emphasize that the values 
$\left.V_k\right|_{\partial \Omega} \, \text{and} \, \left.\left(\J_l \cdot \mathbf{n}\right)\right|_{\partial \Omega} = \left.-\sigma\frac{\partial V_l}{\partial n}\right|_{\partial\Omega}$
are known from our measurement. 
	
In two dimensions, we have the additional null Lagrangians
\begin{equation}\label{nullLagrangians2-D}
  \langle \E_1 \cdot \Rp \E_2 \rangle = \dfrac{1}{\left|\Omega\right|} \displaystyle\int_{\partial \Omega} V_1 \dfrac{\partial V_2}{\partial \mathbf{t}} \quad \text{and} \quad \langle \J_1 \cdot \Rp \J_2 \rangle = -\dfrac{1}{\left|\Omega\right|} \displaystyle\int_{\partial \Omega} \left[\left(\J_1\cdot\mathbf{n}\right) \displaystyle\int_{\x_0}^{\x} \left(\J_2 \cdot \mathbf{n}\right)\right] ,
\end{equation}
where  $\Rp$ is the $2 \times 2$ matrix for a $90\,^{\circ}$ clockwise rotation, namely 
\begin{equation}\label{Rperp}
  \Rp = \begin{bmatrix} 0 & 1 \\ -1 & 0 \end{bmatrix} ,
\end{equation}
 $\mathbf{t} = -\Rp \mathbf{n} = \Rp^T \mathbf{n}$ is the unit tangent vector to $\partial \Omega$, $\frac{\partial V_2}{\partial \mathbf{t}} = \nabla V_2 \cdot \mathbf{t}$, $\x_0 \in \partial \Omega$ is arbitrary, $\x \in \partial \Omega$, and all of the integrals over $\partial \Omega$ are taken in the positive (counterclockwise) orientation.   The first formula in \eqref{nullLagrangians2-D} is found by integration by parts while the derivation of the second formula can be found in \cite{Kang:2012:SBV}.  We note that if the material under consideration is a periodic composite, it is well known that \eqref{nullLagrangians} and \eqref{nullLagrangians2-D} become
\begin{equation}\label{nullLagrangianscomposite}
    \langle \E_k \cdot \J_l \rangle = \langle \E_k \rangle \cdot \langle \J_l \rangle, \quad
    \langle \E_1 \cdot \Rp \E_2 \rangle = \langle \E_1 \rangle \cdot \Rp \langle \E_2 \rangle, \quad \text{and} \quad
    \langle \J_1 \cdot \Rp \J_2 \rangle = \langle \J_1 \rangle \cdot \Rp \langle \J_2 \rangle .
\end{equation}

\subsection{Main Idea}
For $\x \in \Omega, \mathbf{c}^{(\alpha)} \in \mathbb{R}^2$, and $\alpha = 1,2$ we define 
\begin{equation}\label{g}
  \ga(\x;\mathbf{c}^{(\alpha)}) := \displaystyle\sum_{m=1}^2 c^{(\alpha)}_m \left(\E_m^{(\alpha)}(\x) - \dfrac{\ca(\x)}{\fa}\langle\E_m^{(\alpha)}\rangle\right) ,
\end{equation}
where $\E_m^{(\alpha)}(\x) = \ca(\x)\E_m(\x)$.
Note that $\langle\ga\rangle = 0$ for all $\mathbf{c}^{(\alpha)} \in \mathbb{R}^2$.  We must also have $\langle\ga\cdot\ga\rangle\ge 0$ for all $\mathbf{c}^{(\alpha)} \in \mathbb{R}^2$; a computation shows that this is equivalent to requiring
\begin{equation}\label{star}
  \mathbf{c}^{(\alpha)}\cdot \Sa \mathbf{c}^{(\alpha)} \ge 0 \text{ for all } \mathbf{c}^{(\alpha)} \in \mathbb{R}^2  ,
\end{equation}
where
\begin{align}
  \Sa &= \begin{bmatrix} A_{11}^{(\alpha)}-\dfrac{\langle\E_1^{(\alpha)}\rangle\cdot\langle\E_1^{(\alpha)}\rangle}{\fa} & A_{12}^{(\alpha)}-\dfrac{\langle\E_1^{(\alpha)}\rangle\cdot\langle\E_2^{(\alpha)}\rangle}{\fa} \myspace
  				   A_{21}^{(\alpha)}-\dfrac{\langle\E_2^{(\alpha)}\rangle\cdot\langle\E_1^{(\alpha)}\rangle}{\fa} & A_{22}^{(\alpha)}-\dfrac{\langle\E_2^{(\alpha)}\rangle\cdot\langle\E_2^{(\alpha)}\rangle}{\fa} 
	   \end{bmatrix}  \label{Salpha}
  \intertext{and}
  A_{mn}^{(\alpha)} &= \langle\E_m^{(\alpha)}\cdot\E_n^{(\alpha)}\rangle \quad (\text{for} \,\, \alpha,m,n = 1,2) \label{Amnalpha}  .
\end{align}
The matrix $\Sa$ is symmetric by \eqref{Salpha}-\eqref{Amnalpha}; it must also be positive-semidefinite by \eqref{star}.  

We note that the quantities $\langle\E_m^{(\alpha)}\rangle$ are known; this can be seen as follows.  Since $\J = \sigma \E$, 
\begin{equation*}
  \J_1 = \sigma_1\E_1 - \sigma_2\E_2 \quad \text{and} \quad \J_2 = \sigma_2\E_1 + \sigma_1\E_2 .
\end{equation*}
For a field $\mathbf{u}$, we can ``split'' its average value over $\Omega$ into two parts as follows:
\begin{equation}\label{splitting}
  \langle\mathbf{u}\rangle = \langle\chione\mathbf{u}\rangle + \langle\chitwo\mathbf{u}\rangle .
\end{equation}
Note that the averages in \eqref{splitting} are taken over $\Omega$; in particular $\langle\chi^{(1)}\mathbf{u}\rangle$ is not the average of $\mathbf{u}$ over phase 1.
We apply this ``splitting method'' to $\E$ and $\J$ and recall that the conductivity is homogeneous in each phase to obtain the system
\begin{equation*}
  \langle\E^{(1)}\rangle + \langle\E^{(2)}\rangle = \langle\E\rangle \qquad \text{and} \qquad \sigma^{(1)}\langle\E^{(1)}\rangle + \sigma^{(2)}\langle\E^{(2)}\rangle = \langle\J\rangle,
\end{equation*}
which is easily solved for $\langle\E^{(1)}\rangle$ and $\langle\E^{(2)}\rangle$:
\begin{equation}\label{eq:E1E2}
  \langle\E^{(1)}\rangle = \frac{\sigma^{(2)}\langle\E\rangle-\langle\J\rangle}{\sigma^{(2)}-\sigma^{(1)}} \qquad \text{and} \qquad
  \langle\E^{(2)}\rangle = \frac{-\sigma^{(1)}\langle\E\rangle+\langle\J\rangle}{\sigma^{(2)}-\sigma^{(1)}}.
\end{equation}
Since $\langle\E\rangle$ and $\langle\J\rangle$ are known, the real and imaginary parts of $\langle\E^{(1)}\rangle$ and $\langle\E^{(2)}\rangle$ can be determined from \eqref{eq:E1E2} by equating the real and imaginary parts of the left- and right-hand sides of each equation.

In a similar manner, for $k,l = 1, 2$ we have
\begin{equation}\label{power values}
  \langle\E_k\cdot\J_l\rangle = \langle\chione\E_k\cdot\J_l\rangle + \langle\chitwo\E_k\cdot\J_l\rangle .
\end{equation}
The equations in \eqref{power values} are equivalent to the linear system
\begin{equation}\label{powersystem}
  \begin{bmatrix} \sigma_1^{(1)} & \sigma_1^{(2)} & -\sigma_2^{(1)} & -\sigma_2^{(2)} & 0 & 0 \\[0.1cm]
  			  \sigma_2^{(1)} & \sigma_2^{(2)} & \sigma_1^{(1)} & \sigma_1^{(2)} & 0 & 0\\[0.1cm]
			  0 & 0 & \sigma_1^{(1)} & \sigma_1^{(2)} & -\sigma_2^{(1)} & -\sigma_2^{(2)} \\[0.1cm]
			  0 & 0 & \sigma_2^{(1)} & \sigma_2^{(2)} & \sigma_1^{(1)} & \sigma_1^{(2)}
  \end{bmatrix}
  \begin{bmatrix}  A_{11}^{(1)} \\[0.1cm] A_{11}^{(2)} \\[0.1cm] A_{21}^{(1)} \\[0.1cm] A_{21}^{(2)} \\[0.1cm] A_{22}^{(1)} \\[0.1cm] A_{22}^{(2)} \end{bmatrix}
  =
  \begin{bmatrix} \langle \E_1\cdot\J_1\rangle \\[0.1cm] \langle\E_1\cdot \J_2\rangle \\[0.1cm] \langle \E_2\cdot \J_1 \rangle \\[0.1cm] \langle \E_2 \cdot \J_2 \rangle \end{bmatrix} .
\end{equation}
Recall that the right-hand side of this system is known from our measurement (see \eqref{nullLagrangians}).  Since this is an underdetermined system with infinitely-many solutions, we set $x:=A_{11}^{(1)}$ and $y:=A_{11}^{(2)}$ and solve the system \eqref{powersystem} in terms of the ``free variables'' $x$ and $y$.  In particular, we solve the system
\begin{equation}\label{reducedsystem}
  \begin{bmatrix} -\sigma_2^{(1)} & -\sigma_2^{(2)} & 0 & 0 \\[0.1cm]
  			    \sigma_1^{(1)} & \sigma_1^{(2)} & 0 & 0\\[0.1cm]
			    \sigma_1^{(1)} & \sigma_1^{(2)} & -\sigma_2^{(1)} & -\sigma_2^{(2)} \\[0.1cm]
			    \sigma_2^{(1)} & \sigma_2^{(2)} & \sigma_1^{(1)} & \sigma_1^{(2)}
  \end{bmatrix}
  \begin{bmatrix}  A_{21}^{(1)} \\[0.1cm] A_{21}^{(2)} \\[0.1cm] A_{22}^{(1)} \\[0.1cm] A_{22}^{(2)} \end{bmatrix}
  =
  \begin{bmatrix} \lla \E_1\cdot\J_1\rra - \sigma_1^{(1)} x - \sigma_1^{(2)} y \\[0.1cm] \lla \E_1\cdot \J_2\rra -\sigma_2^{(1)} x - \sigma_2^{(2)} y \\[0.1cm] \lla \E_2\cdot \J_1 \rra \\[0.1cm] \lla \E_2 \cdot \J_2 \rra \end{bmatrix}.
\end{equation}
The system \eqref{reducedsystem} has a unique solution if and only if $\beta:=\dets \ne 0$, so for the remainder of this paper we assume that $\beta \ne 0$.  

\begin{betaremark}  We chose $x = A_{11}^{(1)}$ and $y = A_{11}^{(2)}$ arbitrarily.  We could have taken $x = A_{mn}^{(\alpha)}$ for $\alpha, m, n$ either 1 or 2 and $y = A_{mn}^{(\alpha)}$ such that $y \ne x$.  In any of these cases, we would still have arrived at an underdetermined system like that in \eqref{powersystem} that has a unique solution if and only if $\beta \ne 0$, so the condition that $\beta \ne 0$ is independent of how $x$ and $y$ are defined.  
\end{betaremark}

\begin{betaremark2}  The requirement that $\beta \ne 0$ implies that the results of this paper cannot be applied if $\sigma^{(1)}$ and $\sigma^{(2)}$ are both real.  
\end{betaremark2}

Using Maple, we solve \eqref{reducedsystem} in terms of $x$ and $y$, insert the results into the matrices $S^{(1)}$ and $S^{(2)}$ (see \eqref{Salpha}), and replace $f^{(1)}$ by a \emph{test value} $f$.  Denoting the resulting matrices by $S_f^{(1)}$ and $S_f^{(2)}$ we find
\begin{equation}\label{Sxy}
  \begin{aligned}
    S^{(1)}_f(x,y) &:= \begin{bmatrix} x - \dfrac{\|\langle\Eoo\rangle\|^2}{f} & S_{21}^{(1)}(x,y,f) \\ S_{21}^{(1)}(x,y,f) & -x + \etaone - \dfrac{\|\langle\Eto\rangle\|^2}{f} \end{bmatrix} \\[0.1cm]
    S^{(2)}_f(x,y) &:= \begin{bmatrix} y - \dfrac{\|\langle\Eot\rangle\|^2}{1-f} & S_{21}^{(2)}(x,y,f) \\ S_{21}^{(2)}(x,y,f) & -y + \etatwo - \dfrac{\|\langle\Ett\rangle\|^2}{1-f} \end{bmatrix} \\
  \end{aligned}
\end{equation}
for $f \in (0,1)$, where
\begin{equation}\label{constants}
  \begin{aligned}
S_{21}^{(1)}&(x,y,f) = -\gamma x - \psione y + \xione - \dfrac{\langle\Eoo\rangle\cdot\langle\Eto\rangle}{f} ;\\
S_{21}^{(2)}&(x,y,f) =  \psitwo x + \gamma y  - \xitwo - \dfrac{\langle\Eot\rangle\cdot\langle\Ett\rangle}{1-f} ; \\
\beta &= \sigma_1^{(1)}\sigma_2^{(2)} - \sigma_2^{(1)}\sigma_1^{(2)} ; \quad
\gamma = \frac{\sigma_1^{(1)} \sigma_1^{(2)} + \sigma_2^{(1)} \sigma_2^{(2)}}{\beta} ; \quad
\psione = \frac{\left|\sigma^{(2)}\right|^2}{\beta} ; \quad 
\psitwo = \frac{\left|\sigma^{(1)}\right|^2}{\beta} ;\\
\xione  &= \dfrac{\sigma_2^{(2)} \lla\E_1\cdot\J_2\rra + \sigma_1^{(2)} \lla\E_1\cdot \J_1\rra}{\beta} ; \quad
\xitwo = \dfrac{\sigma_2^{(1)} \lla\E_1\cdot\J_2\rra + \sigma_1^{(1)} \lla\E_1\cdot \J_1\rra}{\beta} ; \\
\etaone &= \dfrac{\sigma_1^{(2)}\left(\lla\E_2\cdot\J_1\rra - \lla\E_1\cdot\J_2\rra\right) + \sigma_2^{(2)}\left(\lla\E_1\cdot\J_1\rra + \lla\E_2\cdot\J_2\rra\right)}{\beta} ; \\
\etatwo &= \dfrac{\sigma_1^{(1)} \left(\lla\E_1\cdot\J_2\rra - \lla\E_2\cdot\J_1\rra\right) - \sigma_2^{(1)}\left(\lla\E_1\cdot\J_1\rra + \lla\E_2\cdot\J_2\rra\right)}{\beta} \, .
  \end{aligned}
\end{equation}
Note that $\beta, \gamma, \psione, \psitwo, \xione, \xitwo, \etaone,$ and $\etatwo$ are known.  Moreover, we can use the relationship $\J = \sigma \E$ to rewrite $\eta^{(\alpha)}$ as
\begin{equation}\label{etarewrite}
  \eta^{(\alpha)} = \langle\ca\left(\|\E_1\|^2 + \|\E_2\|^2\right)\rangle = \langle\|\E^{(\alpha)}_1\|^2\rangle + \langle\|\E^{(\alpha)}_2\|^2\rangle \, .
\end{equation}
Note that $\eta^{(\alpha)} \ge 0$ with equality if and only if $\E^{(\alpha)} = \E^{(\alpha)}_1 + \mathrm{i}\E^{(\alpha)}_2 \equiv 0$ (up to a set of measure 0); that is, $\eta^{(\alpha)} = 0$ if and only if the electric field is 0 in phase $\alpha$.  In two dimensions with $D$ having smooth boundary the condition that the field is zero in one phase implies that it is zero everywhere; thus $\eta^{(\alpha)} = 0$ only for trivial boundary conditions.  In three dimensions the situation is less clear \cite{Alessandrini:2009:SCP}, but in practice the field will almost always be zero in one of the phases only for trivial boundary conditions.  Therefore we assume throughout the rest of this paper that $\etaone \ne 0$ and $\etatwo \ne 0$.

\begin{feasibleregion}\label{feasibleregion}
  For $f \in (0,1)$ we set \[\Fa_f := \{(x,y) \in \mathbb{R}^2 : \Sa_f(x,y) \text{ is positive-semidefinite}\} .\]
  Then the set $\F_f := \Fo_f \cap \Ft_f$ is called the \emph{feasible region associated with $f$}.  
  In addition, the set $\A := \left\{f \in (0,1) : \F_f \ne \emptyset\right\}$ is called the \emph{set of admissible test values}.   
\end{feasibleregion}
Practically, given $f \in \left(0,1\right)$, we check to see whether or not there are regions in the $xy-$plane for which $S^{(1)}_f(x,y)$ and $S^{(2)}_f(x,y)$ are simultaneously positive-semidefinite--that is, whether or not $\F_f \ne \emptyset$.  If the feasible region $\F_f$ is nonempty, then $f$ is an admissible test value, so $f \in \A$; that is, $f$ \emph{may} be the true volume fraction of phase 1.  If $\F_f = \emptyset$ we can conclude that $f$ is \emph{not} the true volume fraction of phase 1.  This argument is based on the fact that $\F_{f^{(1)}}$ cannot be empty, by \eqref{star}.

Our goal is to find the set $\A$.  If $\A$ is connected, the desired lower and upper bounds on $\fone$ will be $\inf \A$ and $\sup \A$, respectively.  If $\A$ is not connected, the structure of the bounds will be more complicated--see Figure \ref{boundsinterval}.  In Figure \ref{disconnected_fig}, the set of admissible test values is $\A= \A^* \cup \A^{**}$.  In the examples we have encountered $\A$ has always been connected.
\begin{figure}[!hbtp]
  \centering
  \begin{subfigure}{\textwidth}
    \centering
    \includegraphics{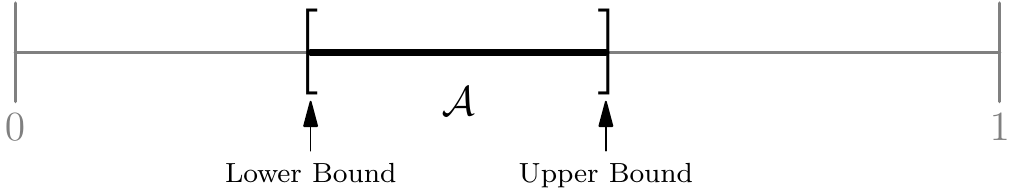}
    \caption{$\A$ connected}
    \label{connected_fig}
  \end{subfigure}
  
  \begin{subfigure}{\textwidth}
    \centering
    \includegraphics{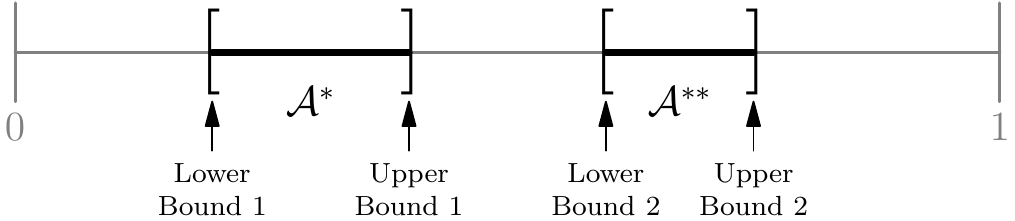}
    \caption{$\A$ disconnected}
    \label{disconnected_fig}
  \end{subfigure}
  \caption{\textit{(a) When $\A$ (the darkened interval) is connected, we have $\inf \A \le \fone \le \sup \A$.  (b) When $\A = \A^* \cup \A^{**}$ is disconnected, there will be multiple bounds on $\fone$.  In the example above, we know that either $\inf \A^* \le \fone \le \sup \A^*$ or $\inf \A^{**} \le \fone \le \sup \A^{**}$.}}  
  \label{boundsinterval}
\end{figure}


\section{Elementary Bounds}\label{section_elementary_bounds}
Recall that a symmetric $2\times 2$ matrix 
\begin{equation*}
  L = \begin{bmatrix} a & b \\ b & c \end{bmatrix}
\end{equation*}
is positive-semidefinite if and only if $a \ge 0, c \ge 0,$ and $ac-b^2 = \det L \ge 0$.  In this section we use the above requirements on the diagonal components of the matrices $S^{(1)}_f(x,y)$ and $S^{(2)}_f(x,y)$ to derive elementary bounds on $\fone$.  

By Definition \ref{feasibleregion} and the above statement, $f \in \A$ only if there is at least one point $(x,y) \in \mathbb{R}^2$ such that $\Sa_{f,mm}(x,y) \ge 0$ for $\alpha, m = 1,2$.  That is, the following inequalities must hold for all admissible volume fractions $f$ (see \eqref{Sxy}):
\begin{subequations}
  \begin{align}
    &\dfrac{\|\langle\Eoo\rangle\|^2}{f} \le x \le \etaone - \dfrac{\|\langle\Eto\rangle\|^2}{f} \label{xbounds}\\
    &\dfrac{\|\langle\Eot\rangle\|^2}{1-f} \le y \le \etatwo - \dfrac{\|\langle\Ett\rangle\|^2}{1-f} \, .\label{ybounds}
  \end{align}
\end{subequations}
\begin{elementaryfeasibleregion}\label{elementaryfeasibleregion}
  For $f \in (0,1)$, the set $\F_{f,e} := \left\{(x,y) \in \mathbb{R}^2 : \text{both } \eqref{xbounds} \text{ and } \eqref{ybounds} \text{ hold}\right\}$ is called the \emph{elementary feasible region associated with $f$}.  The set $\A_e := \left\{f \in (0,1) : \F_{f,e} \ne \emptyset\right\}$ is called the \emph{elementary set of admissible test values}.  
\end{elementaryfeasibleregion}
Geometrically, for each admissible $f \in (0,1)$, the set $\F_{f,e}$ will be the closed rectangle in $\mathbb{R}^2$ defined by the inequalities in \eqref{xbounds} and \eqref{ybounds}.
For a given $f \in (0,1)$, the set $\F_{f,e}$ will be nonempty if and only if both of the following inequalities hold.  
\begin{subequations}
  \begin{align}
    \dfrac{\|\langle\Eoo\rangle\|^2}{f} &\le \etaone-\dfrac{\|\langle\Eto\rangle\|^2}{f} \label{xs} \\
    \dfrac{\|\langle\Eot\rangle\|^2}{1-f} &\le \etatwo-\dfrac{\|\langle\Ett\rangle\|^2}{1-f} \label{ys} \, .
  \end{align}
\end{subequations}

As stated earlier we assume that $\eta^{(\alpha)} \ne 0$ ($\Leftrightarrow \E^{(\alpha)} \not\equiv 0$) for $\alpha = 1, 2$.  Then the inequalities in \eqref{xs} and \eqref{ys} may be rewritten as 
\begin{subequations}
  \begin{align}
    f &\ge f_{e,l} := \dfrac{\|\langle\Eoo\rangle\|^2 + \|\langle\Eto\rangle\|^2}{\etaone} \label{lowere}\\
    f &\le f_{e,u} := 1-\dfrac{\|\langle\Eot\rangle\|^2 + \|\langle\Ett\rangle\|^2}{\etatwo} \label{uppere},
  \end{align}
\end{subequations}
so $\A_e = \left[f_{e,l},f_{e,u}\right]$.  
We obtain elementary bounds on $\fone$ by combining \eqref{lowere} and \eqref{uppere} and noting that $\fone$ must be in $\A_e$:
\begin{equation}\label{elementarybounds}
  f_{e,l} \le \fone \le f_{e,u} \, .
\end{equation}
We emphasize that $f_{e,l}$ and $f_{e,u}$ can be computed from the boundary measurements--see \eqref{eq:E1E2} and \eqref{constants}.
Note that $0 \le f_{e,l}$ and $f_{e,u} \le 1$.  Since
\begin{equation}\label{wbounds}
  \lla\left\|\E_m^{(\alpha)} - \frac{\ca}{\fa}\lla\E_m^{(\alpha)}\rra\right\|^2\rra \ge 0
  \quad \Leftrightarrow \quad \|\langle\E^{(\alpha)}_m\rangle\|^2 \le \fa \langle\|\E^{(\alpha)}_m\|^2\rangle, 
\end{equation}
we have $f_{e,l} - f_{e,u} \le \fone + \ftwo -1 = 0$ and so $f_{e,l} \le f_{e,u}$.  

We also note that the previous sentence leads to a simpler proof of the elementary bounds.    In particular, \eqref{wbounds} implies that 
\[\|\langle\E^{(\alpha)}_1\rangle\|^2 + \|\langle\E^{(\alpha)}_2\rangle\|^2 \le f^{(\alpha)}\left[\langle\|\E^{(\alpha)}_1\|^2\rangle + \langle\|\E^{(\alpha)}_2\|^2\rangle\right].\]
The first and second inequalities in \eqref{elementarybounds} follow from this by taking $\alpha = 1$ and $\alpha = 2$, respectively (recall $f^{(2)} = 1-f^{(1)}$).  

Now \eqref{wbounds} holds as an equality if and only if 
\begin{equation*}
  \E^{(\alpha)}_m(\x) = \ca(\x) \frac{\langle\E_m^{(\alpha)}\rangle}{\fa};
\end{equation*}
that is, \eqref{wbounds} holds as an equality if and only if $\E_m$ is a constant in phase $\alpha$.  From this we see that $f_{e,l} = \fone$ if and only if $\E^{(1)} = \chione\E$ is a constant (which must be nonzero since we are assuming $\etaone \ne 0 \Leftrightarrow \E^{(1)} \not \equiv 0$) and $f_{e,u} = \fone$ if and only if $\E^{(2)} = \chitwo\E$ is a (nonzero) constant.  This implies that the bounds in \eqref{elementarybounds} are sharp in the sense that the lower bound (upper bound) is satisfied as an equality for geometries in which the electric field is constant in phase 1 (phase 2).  

For example, if phase 1 is a disk of radius $r$ centered at the origin and phase 2 is a concentric disk of radius $R > r$, then $\E^{(1)}$ will be a constant for the affine Dirichlet boundary condition $V_0 = \mathbf{u}\cdot\x$, where $\mathbf{u} \ne 0 \in \mathbb{C}^2$.  In this case $f_{e,l} = f^{(1)}$.  If we relabel the phases then $\E^{(2)}$ will be a constant, so $f_{e,u} = f^{(1)}$.  A simple laminate of materials with conductivities $\sigma^{(1)}$ and $\sigma^{(2)}$ has the property that the electric field is constant in both phases, so $f_{e,l} = f_{e,u} = f^{(1)}$ in that case.  In 2-D there are many examples of inclusions inside which the electric field is constant for certain boundary conditions--see \cite{Kang:2012:SBV} for elegant constructions of these so-called $E_{\Omega}$ inclusions.  Although the argument in \cite{Kang:2012:SBV} was applied in the real conductivity case, it extends to the complex conductivity case as well.  So for appropriate boundary conditions the field inside an $E_{\Omega}$ inclusion will be uniform even when the conductivities are complex.
We have thus proven the following theorem.  
\begin{elementary_bounds_theorem}[Elementary Bounds]\label{elementary_bounds_theorem}
  Assume that $\beta \ne 0$ (where $\beta$ is defined in \eqref{constants}).  If $\eta^{(\alpha)} \ne 0$ ($\Leftrightarrow \E^{(\alpha)} \not\equiv 0$) for $\alpha = 1, 2$, then $f_{e,l} \le \fone \le f_{e,u}$.  Moreover, $f_{e,l} = \fone$ if and only if $\E^{(1)}$ is a nonzero constant and $f_{e,u} = \fone$ if and only if $\E^{(2)}$ is a nonzero constant. 
\end{elementary_bounds_theorem}
In particular, this theorem states that
\begin{equation}\label{Ae}
  \A_e = \left[f_{e,l},f_{e,u}\right].
\end{equation}

We illustrate these ideas by considering an example, shown in Figure \ref{elementaryboundsfigure}.  
We consider an annular ring with conductivity $\sigma^{(2)}$ and a discontinuous ``inclusion phase'' $D$ consisting of the core and surrounding material outside the annulus with conductivity $\sigma^{(1)}$.  Figure \ref{annulus_fig} is a sketch of the region $\Omega$.   In Figure \ref{xybounds_fig} we plot the bounds from \eqref{xbounds} and \eqref{ybounds} versus $f$.  In particular, the lower bound in \eqref{xbounds} is plotted as a red dashed line while the upper bound is plotted as a red solid line. The red shaded region indicates the values of $f$ for which the bounds in \eqref{xbounds} hold, i.e. the values of $f$ for which there is at least one value of $x$ such that \eqref{xbounds} holds.  Similarly, the lower bound in \eqref{ybounds} is plotted as a blue dash-dotted line while the upper bound is plotted as a blue dotted line.  The blue shaded region indicates the values of $f$ for which there is at least one value of $y$ such that the bounds in \eqref{ybounds} hold.  The left and right black vertical lines indicate the elementary lower and upper bounds $f_{e,l}$ and $f_{e,u}$, respectively; the dashed magenta line indicates the true volume fraction $f^{(1)}$. The elementary set of admissible test values, $\A_e$, is indicated by the darkened interval between $f_{e,l}$ and $f_{e,u}$.  
\begin{figure}[!hbtp]
  \centering
  \begin{subfigure}{0.4\textwidth}
    \centering
    \includegraphics[scale=0.8]{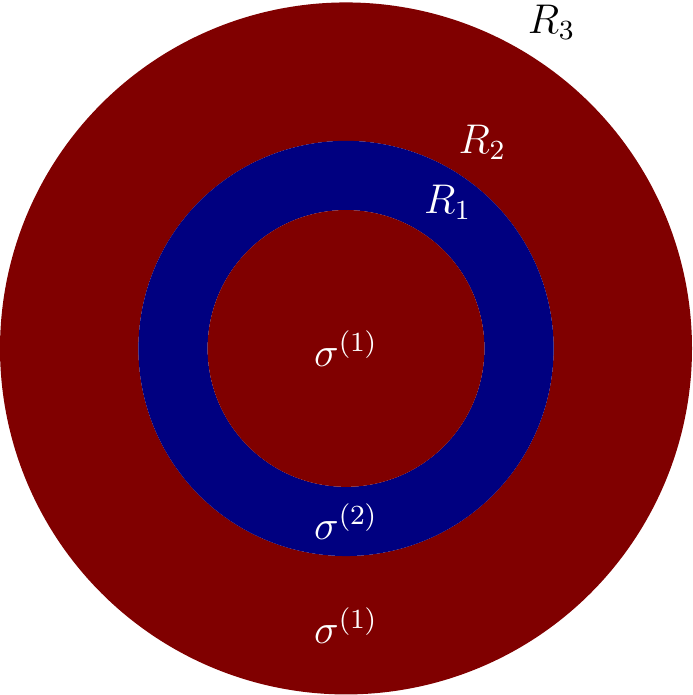}
    \caption{}
    \label{annulus_fig}
  \end{subfigure}
  \qquad \qquad
  \begin{subfigure}{0.4\textwidth}
    \centering
    \includegraphics[scale=0.8]{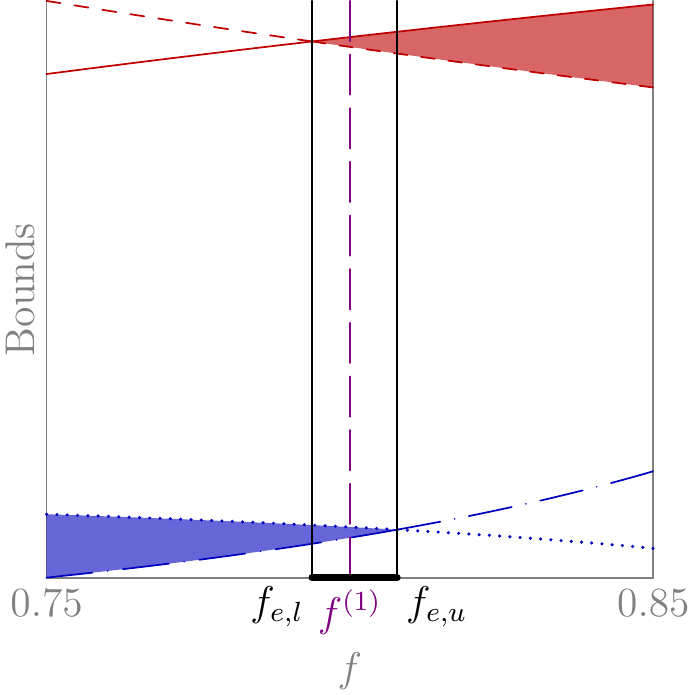}
    \caption{}
    \label{xybounds_fig}
  \end{subfigure}
  \caption{\emph{(a) A sketch of the region under consideration--our discontinuous ``inclusion phase'' $D$ (with conductivity $\sigma^{(1)}$ and volume fraction $f^{(1)}$) is the core plus the surrounding material outside the annulus.  (b) Construction of the elementary bounds.  The parameters that were used to create these plots are: radii $R_1 = 2$; $R_2 = 3$; $R_3 = 5$; conductivities $\sigma^{(1)} = 3+8\mathrm{i}$; $\sigma^{(2)} = 8+6\mathrm{i}$; the Dirichlet boundary condition was $V_0 = \mathbf{u}\cdot\mathbf{x}$, where $\mathbf{u} = \left(-2+\mathrm{i}, \frac{3}{5}-\frac{7}{5}\mathrm{i}\right)^T$.  The elementary lower and upper bounds are $f_{e,l} \approx 0.794$ and $f_{e,u} \approx 0.808,$ respectively.  The true volume fraction is $f^{(1)} = 0.8$.}}  \label{elementaryboundsfigure}
\end{figure}


\section{More Sophisticated Bounds}\label{section_more_sophisticated_bounds}
Throughout this section, we assume that $\etaone$ and $\etatwo$ are both nonzero.  We derive a method to determine bounds by using the additional requirement that $\Sa_f(x,y)$ is positive-semidefinite only if $\det \Sa_f(x,y) \ge 0$.  Using \eqref{Sxy} we find, for $\alpha = 1, 2$,
\begin{equation}\label{ps}
  \pa_f(x,y) := \det \Sa_f(x,y) = \al_1 x^2 + 2\al_2 xy + \al_3 y^2 + 2\al_4 x + 2\al_5 y + \al_6 
\end{equation}
where
\begin{equation}\label{aones}
  \left\{\begin{aligned} 
    \aone_1 &= -(1+\gamma^2) ; \quad
    \aone_2 = -\gamma\psione ; \quad 
    \aone_3 = -\left[\psione\right]^2; \\
    \aone_4 &= \dfrac{1}{2}\left\{\etaone - \frac{\|\langle\Eto\rangle\|^2}{f} + \frac{\|\langle\Eoo\rangle\|^2}{f} + 2\gamma \left[\xione - \frac{\langle\Eoo\rangle\cdot\langle\Eto\rangle}{f}\right]\right\} ; \\
    \aone_5 &= \psione\left[\xione - \frac{\langle\Eoo\rangle\cdot\langle\Eto\rangle}{f}\right] ; \\
    \aone_6 &= -\left\{\frac{\|\langle\Eoo\rangle\|^2}{f}\left[\etaone - \frac{\|\langle\Eto\rangle\|^2}{f}\right] + \left[\xione - \frac{\langle\Eoo\rangle\cdot\langle\Eto\rangle}{f}\right]^2\right\}
  \end{aligned} \right.
\end{equation}
and
\begin{equation}\label{atwos}
  \left\{\begin{aligned}
  \atwo_1 &= -\left[\psitwo\right]^2 ; \quad
  \atwo_2 = -\gamma\psitwo ; \quad
  \atwo_3 = -\left(1+\gamma^2\right); \\
  \atwo_4 &= \psitwo\left[\xitwo + \frac{\langle\Eot\rangle\cdot\langle\Ett\rangle}{1-f}\right] ; \\
  \atwo_5 &= \dfrac{1}{2}\left\{\etatwo - \frac{\|\langle\Ett\rangle\|^2}{1-f} + \frac{\|\langle\Eot\rangle\|^2}{1-f} + 2\gamma \left[\xitwo + \frac{\langle\Eot\rangle\cdot\langle\Ett\rangle}{1-f}\right]\right\} ; \\
  \atwo_6 &= -\left\{\frac{\|\langle\Eot\rangle\|^2}{1-f}\left[\etatwo - \frac{\|\langle\Ett\rangle\|^2}{1-f}\right] + \left[\xitwo + \frac{\langle\Eot\rangle\cdot\langle\Ett\rangle}{1-f}\right]^2\right\}.
  \end{aligned}\right.
\end{equation}
\begin{ellipseregion}\label{ellipseregion}
  For $\alpha = 1, 2$ and for $f \in \A_e \,(= [f_{e,l},f_{e,u}])$ we define
  \begin{equation*}
    \Ea_f := \{\left(x,y\right) \in \mathbb{R}^2 : \pa_f(x,y) \ge 0\} \quad \text{and} \quad \cE_f := \Ei .
  \end{equation*}  
\end{ellipseregion}
We will now prove several lemmas in order to establish some useful properties of the sets $\Ea_f$.
\begin{ellipseslemma}\label{ellipseslemma}
  Assume that $\beta \ne 0$ and $\eta^{(\alpha)} \ne 0$ for $\alpha = 1, 2$.  Then the following properties hold.
  \begin{enumerate}
    \item[(1)] For $f \in \left(f_{e,l}, f_{e,u}\right)$ and $\alpha = 1, 2$, $\Ea_f$ is a closed elliptic disk; its boundary is the ellipse $\partial\Ea_f = \{(x, y) \in \mathbb{R}^2 : \pa_f(x,y) = 0\}$;  
    \item[(2)] $\cE^{(1)}_{f_{e,l}}$ is a point and $\cE^{(2)}_{f_{e,l}}$ is a closed elliptic disk;
    \item[(3)] $\cE^{(1)}_{f_{e,u}}$ is a closed elliptic disk and $\cE^{(2)}_{f_{e,u}}$ is a point.
  \end{enumerate}
\end{ellipseslemma}
\begin{proof}
  The discriminant of $\pa_f$ is $\al_1\al_3 - \left[\al_2\right]^2 = \left[\psi^{(\alpha)}\right]^2 > 0$ for all $f \in \A_e$.
  Thus the graph of $\pa_f$ is an elliptic paraboloid for all $f \in \A_e$.  The Hessian matrix of $\pa_f$ is 
  \begin{equation*}
    H^{(\alpha)}_f := \begin{bmatrix} 2\al_1 & 2\al_2 \\[0.1cm] 2\al_2 & 2\al_3 \end{bmatrix}.
  \end{equation*}  
  By \eqref{constants}, \eqref{aones}, and \eqref{atwos}, $\al_1 < 0$ and $\det H_f^{(\alpha)} = 4\left[\psi^{(\alpha)}\right]^2 > 0$, so $H_f^{(\alpha)}$ is negative-definite for all $f \in \A_e$; thus $\pa_f$ is concave for all $f \in \A_e$.  By Definition \ref{ellipseregion}, therefore, $\Ea_f$ is the intersection of the plane $z = 0$ with the graph of $\pa_f$.  
  
  For $f \in \A_e$ we define
  \begin{equation}\label{pmax}
    \pa_{f,\max} := \displaystyle\max_{(x,y) \in \mathbb{R}^2} \pa_f(x,y) .
  \end{equation}
  Then $\Ea_f$ will be a closed elliptic disk with boundary 
  $\partial\Ea_f = \{(x, y) \in \mathbb{R}^2 : \pa_f(x,y) = 0\}$
  if and only if $\pa_{f, \max} > 0$, a point if and only if $\pa_{f,\max} = 0$, or the empty set if and only if $\pa_{f,\max} < 0$.  Using calculus, we find that the maximum of $\pa_f$ occurs at the point
  \begin{equation}\label{center}
    \mathbf{r}^{(\alpha)}_f := \left(\overline{x}^{(\alpha)}_f, \overline{y}^{(\alpha)}_f\right)
    := \left(\dfrac{\al_2\al_5-\al_3\al_4}{\left[\psi^{(\alpha)}\right]^2}, \dfrac{\al_2\al_4-\al_1\al_5}{\left[\psi^{(\alpha)}\right]^2}\right) .
  \end{equation}
  Then we have
  \begin{equation}\label{pfmax}
    \pa_{f, \max} = \dfrac{1}{4 f_*^2}\left\{\eta^{(\alpha)} f_* - \left[\|\langle\E_1^{(\alpha)}\rangle\|^2 + \|\langle\E_2^{(\alpha)}\rangle\|^2\right]\right\}^2,
   \end{equation}
   where
   \begin{equation}\label{fstar}
      f_* := \begin{cases} f &\text{if } \alpha = 1 \\
  						    1-f &\text{if } \alpha = 2 .
			   \end{cases}
   \end{equation}
  Thus $\pa_{f,\max} \ge 0$ for all $f \in \A_e$; in particular $p^{(1)}_{f,\max} = 0$ if and only if $f = f_{e, l}$ (see \eqref{lowere}) while $p^{(2)}_{f,\max} = 0$ if and only if $f = f_{e, u}$ (see \eqref{uppere}).  Therefore $\Ea_f$ is a closed elliptic disk for $f \in \left(f_{e.l}, f_{e,u}\right)$, $\cE^{(1)}_{f_{e,l}}$ is a point and $\cE^{(2)}_{f_{e,l}}$ is a closed elliptic disk, and $\cE^{(1)}_{f_{e,u}}$ is a closed elliptic disk and $\cE^{(2)}_{f_{e,u}}$ is a point.
\end{proof}

\begin{boundingbox}\label{boundingbox}
  Suppose $\beta \ne 0$ and $\eta^{(\alpha)} \ne 0$ for $\alpha = 1, 2$.  Then for each $f \in \A_e \, (= [f_{e,l}, f_{e,u}])$, $\cE_f \subseteq \F_{f,e}$. 
\end{boundingbox}  
\begin{boundingboxremark}
This Lemma states that, for each $f \in \A_e$, the intersection of the elliptic disks (the set $\cE_f$) is contained in the elementary feasible region associated with $f$ (the set $\F_{f,e})$.  Thus the feasible region associated with $f$ (the set $\F_f$) is simply the set $\cE_f$.  In other words, if the elliptic disks $\cE^{(1)}_f$ and $\cE^{(2)}_f$ intersect so that $\cE_f \ne \emptyset$, then $f \in \A$; if the elliptic disks do not intersect so that $\cE_f = \emptyset$, then $f \notin \A$.
\end{boundingboxremark}
\begin{proof}
  For each $f \in \left[f_{e,l},f_{e,u}\right]$ the set $\F_{f,e}$ contains $\F_f^{(1)}$.  The boundary of the set $\F_f^{(1)}$ is described by the equation $p^{(1)}_f(x,y) = 0$ which, according to Lemma \ref{ellipseslemma}, is either an ellipse, a point, or the empty set.  Therefore $\cE_f^{(1)} = \F_f^{(1)} \subseteq \F_{f,e}$.  A similar argument shows that $\cE_f^{(2)} = \F_f^{(2)} \subseteq \F_{f,e}$. 
\end{proof}
\begin{boundingboxremark}\label{boundingboxremark}
In fact, motivated by \eqref{xbounds} one can show that the ellipse $\partial \cE^{(1)}_f$ is tangent to the boundary of the set
\begin{equation}\label{X}
  X_f := \left\{(x,y) \in \mathbb{R}^2 : \dfrac{\|\langle\Eoo\rangle\|^2}{f} \le x \le \etaone - \dfrac{\|\langle\Eto\rangle\|}{f} \right\}
\end{equation}
for $f \in (f_{e,l},f_{e,u}]$.
Similarly, motivated by \eqref{ybounds} one can also show that the ellipse $\partial \cE^{(2)}_f$ is tangent to the boundary of the set
\begin{equation}\label{Y}
  Y_f := \left\{(x,y) \in \mathbb{R}^2 : \dfrac{\|\langle\Eot\rangle\|^2}{1-f} \le y \le \etatwo - \dfrac{\|\langle\Ett\rangle\|^2}{1-f} \right\}
\end{equation}
for $f \in [f_{e,l},f_{e,u})$.  See Figure \ref{ellipses_fig} for an illustration of this fact.  The set $X_f \cap Y_f$ is in fact the rectangle $\F_{f,e}$ and the test values $f$ where this rectangle collapses to a line segment are the elementary bounds.
\end{boundingboxremark}
\begin{twointersectionpoints}\label{twointersectionpoints}
  Suppose $\beta \ne 0$ and $\eta^{(\alpha)} \ne 0$ for $\alpha = 1, 2$.  Then for each $f \in \A_e$ the set $\partial \cE^{(1)}_f \cap \partial \cE^{(2)}_f$ contains at most 2 points.  
\end{twointersectionpoints}
\begin{proof}
  Fix $f \in \A_e$ and suppose that the point $(x,y) \in \partial \cE^{(1)}_f \cap \partial \cE^{(2)}_f$ (note that $\partial \cE^{(\alpha)}_f \ne \emptyset$ by Lemma \ref{ellipseslemma}).  Then for $\alpha = 1, 2$ we must have $\pa_f(x,y) = 0$, where $\pa_f$ is defined in \eqref{ps}.  This implies that
  \begin{equation}\label{mueq}
    0 = \left|\sigma^{(1)}\right|^2 p^{(1)}_f(x,y) - \left|\sigma^{(2)}\right|^2 p^{(2)}_f(x,y) = \mu_4 x + \mu_5 y + \mu_6 ,
  \end{equation}  
  where $\mu_k = \left|\sigma^{(1)}\right|^2\aone_k - \left|\sigma^{(2)}\right|^2\atwo_k$ for $k = 1, 3,$ and $6$, and $\mu_k = 2\left|\sigma^{(1)}\right|^2\aone_k - 2\left|\sigma^{(2)}\right|^2\atwo_k$ for $k = 2 ,4,$ and $5$.  By \eqref{aones} and \eqref{atwos}, $\mu_1 = \mu_2 = \mu_3 = 0$ for all $f \in \A_e$.  We solve \eqref{mueq} for $y$ to find 
  \begin{equation}\label{yy}
    y = -\dfrac{\mu_4 x + \mu_6}{\mu_5} .
  \end{equation}
  Lemma \ref{boundingbox} implies that $y$ is finite for all $f \in \A_e$.  Inserting \eqref{yy} into the equation $p^{(1)}(x,y) = 0$ we find that $x$ must be a root of the quadratic 
  \begin{equation}\label{q}
    q(x):= \nu_1 x^2 + \nu_2 x + \nu_3 ,
  \end{equation}
  where
  \begin{equation}\label{nus}
    \left\{\begin{aligned}
      \nu_1 &= \aone_1\mu_5^2 - 2\aone_2\mu_4\mu_5 + \aone_3\mu_4^2;\\
      \nu_2 &= 2\left[-\aone_2\mu_5\mu_6+ \aone_3\mu_4\mu_6+\aone_4\mu_5^2-\aone_5\mu_4\mu_5\right];\\
      \nu_3 &= \aone_3\mu_6^2-2\aone_5\mu_5\mu_6 + \aone_6\mu_5^2.
    \end{aligned}\right.
  \end{equation}
  (Note that $\nu_1, \nu_2,$ and $\nu_3$ are all functions of $f$).
  The discriminant of $q$ is 
  \begin{equation}\label{Delta}
    \Delta_f:=\nu_2^2-4\nu_1\nu_3 .
  \end{equation}
  Therefore the set $\partial\cE^{(1)}_f \cap \partial\cE^{(2)}_f$ will be 2 (real) points if $\Delta_f > 0$, 1 (real) point if $\Delta_f = 0$, and 0 (real) points if $\Delta_f < 0$.
\end{proof}
Lemma \ref{ellipseslemma} implies that $\cE^{(1)}_f$ and $\cE^{(2)}_f$ are nonempty for all $f \in \A_e$, and Lemma \ref{boundingbox} implies that $\F_f = \cE_f$ for all $f \in \A_e$.  Therefore $f \in \A$ if $\Delta_f \ge 0$, since $\Delta_f \ge 0$ implies $\cE_f \ne \emptyset$.  If $\Delta_f < 0$, $\cE_f$ may be empty or nonempty.  For example, if one of the elliptic disks is completely inside the other, $\Delta_f < 0$ but $\cE_f \ne \emptyset$.  

To determine whether or not $\cE_f$ is empty when $\Delta_f < 0$ we examine the following 4 possibilities.
\begin{enumerate}
\item[(1)] If $p^{(1)}_f(\mathbf{r}^{(2)}) < 0$ and $p^{(2)}_f(\mathbf{r}^{(1)}) < 0$, then the elliptic disks (which may be points) are disjoint since neither elliptic disk contains the center of the other.  Thus $\cE_f = \emptyset$ which implies that $f \notin \A$;
\item[(2)] if $p^{(1)}_f(\mathbf{r}^{(2)}) \ge 0$ and $p^{(2)}_f(\mathbf{r}^{(1)}) < 0$ then the elliptic disk $\cE^{(1)}_f$ contains the center of the elliptic disk $\cE^{(2)}_f$ but not vice versa.  In this case $\cE^{(2)}_f \subsetneq \cE^{(1)}_f \Rightarrow \cE_f \ne \emptyset \Rightarrow f \in \A$;
\item[(3)] if $p^{(1)}_f(\mathbf{r}^{(2)}) < 0$ and $p^{(2)}_f(\mathbf{r}^{(1)}) \ge 0$ then $\cE^{(1)}_f \subsetneq \cE^{(2)}_f \Rightarrow \cE_f \ne \emptyset \Rightarrow f \in \A$;
\item[(4)] if $p^{(1)}_f(\mathbf{r}^{(2)}) \ge 0$ and $p^{(2)}_f(\mathbf{r}^{(1)}) \ge 0$ we can conclude that $\cE_f \ne \emptyset$ and so $f \in \A$.
\end{enumerate}
Unfortunately $\Delta_f$ is a complicated function of $f$, so it is difficult if not impossible to determine the sign of $\Delta_f$ analytically.  The expressions for $p^{(1)}(\mathbf{r}^{(2)})$ and $p^{(2)}(\mathbf{r}^{(1)})$ are non-trivial as well, so the above steps must be carried out numerically.  (For example, for the configuration considered in Figure \ref{annulus_fig}, $\Delta_f$ is essentially a rational function with an irreducible polynomial of degree 8 in the numerator and an irreducible polynomial of degree 2 in the denominator.  The functions $p^{(1)}_f(\mathbf{r}^{(2)})$ and $p^{(2)}_f(\mathbf{r}^{(1)})$ are rational functions with irreducible polynomials of degree 4 in the numerator).  We have thus proven the following theorem.
\begin{noRperptheorem}\label{noRperptheorem}
  Suppose $\beta \ne 0$ and $\eta^{(\alpha)} \ne 0$ for $\alpha = 1, 2$.  Then for $f \in \A_e \, (= \left[f_{e,l}, f_{e,u}\right])$, if $\Delta_f \ge 0$ then $f \in \A$, where $\Delta_f$ is defined in \eqref{Delta}.  If $\Delta_f < 0$, then $f \notin \A$ if and only if $p^{(1)}_f(\mathbf{r}^{(2)}) < 0$ and $p^{(2)}_f(\mathbf{r}^{(1)}) < 0$.  
\end{noRperptheorem}
As mentioned in the introduction, the bounds derived in this section may or may not be tighter than the elementary bounds from Section \ref{section_elementary_bounds}.  For example, the bounds from this section would be the same as the elementary bounds if $\Delta_f \ge 0$ for all $f \in \A_e$.  We also note that Lemmas \ref{ellipseslemma} and \ref{twointersectionpoints} hold for all $f \in (0,1)$.  This shows the importance of the elementary bounds: if we did not take them into account and only looked at the set $\cE_f$ for all $f \in (0,1)$, it may be that $\cE_f \ne\emptyset$ for all $f \in (0,1)$ (this is indeed the case for the configuration in Figure \ref{elementaryboundsfigure}).  This would only give the trivial bounds $0 < f^{(1)} < 1$.  Although we do not know if this is generally the case, in all of the two dimensional examples we have encountered thus far the ``more sophisticated'' bounds determined using the elliptic disks have been the same as the elementary bounds--see Figures \ref{ellipses_fig} and \ref{numerical_figure}, for example.  So it is not clear if the ``more sophisticated'' bounds are ever better than the elementary bounds.  Irrespective of this, the analysis presented here is useful for the treatment presented in the next section where we do obtain tighter bounds using elliptic disks.  Also, the more sophisticated bounds developed here are beneficial for periodic composite materials, where one may be given the volume fraction and wish to determine bounds on the possible values of the complex pair $(\langle\E\rangle, \langle\J\rangle)$.
\begin{figure}[!hbtp]
  \centering
  \begin{subfigure}{0.2\textwidth}
    \centering
    \includegraphics[scale=0.5]{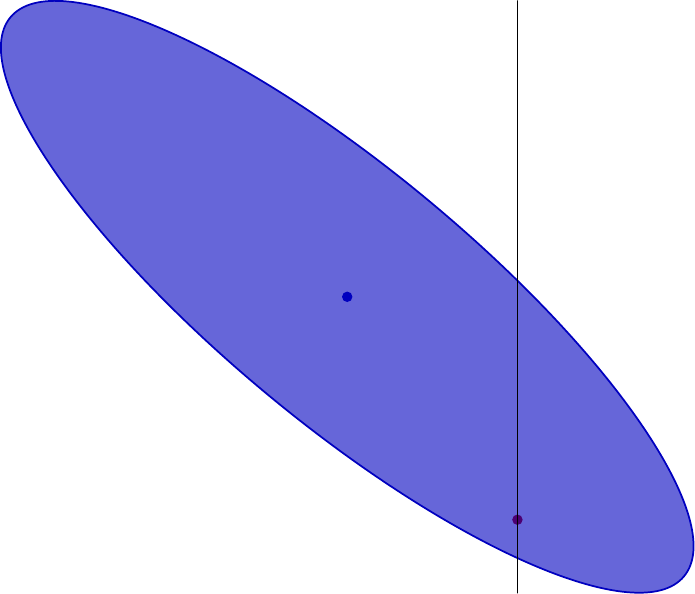}
    \caption{}
    \label{ellipse_static_2_elementary_lower_bound}
  \end{subfigure}
  \qquad
  \begin{subfigure}{0.2\textwidth}
    \centering
    \includegraphics[scale=0.5]{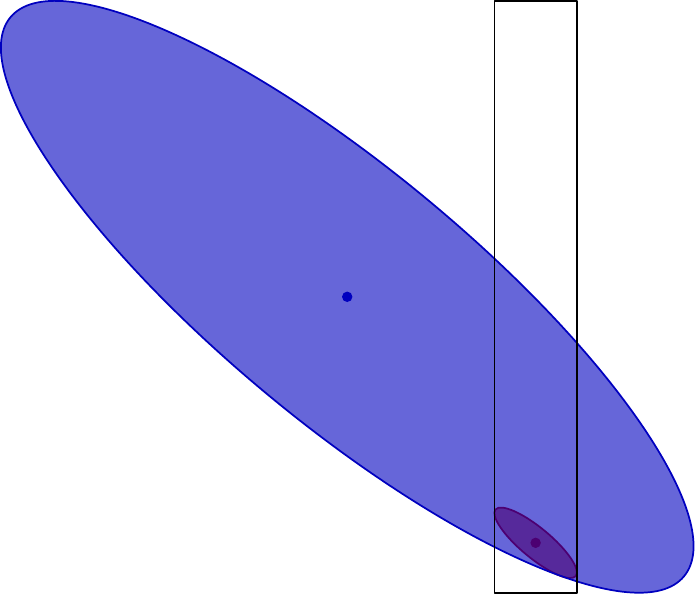}
    \caption{}
    \label{ellipse_static_2_firstdisczero}
  \end{subfigure}
  \qquad
  \begin{subfigure}{0.2\textwidth}
    \centering
    \includegraphics[scale=0.5]{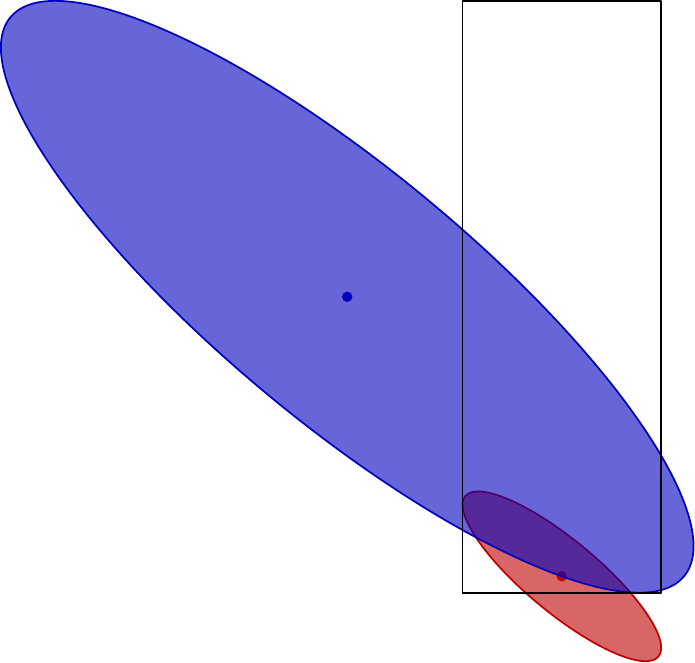}
    \caption{}
    \label{ellipse_static_2_p2c1zero}
  \end{subfigure}
  \qquad
  \begin{subfigure}{0.2\textwidth}
    \centering
    \includegraphics[scale=0.5]{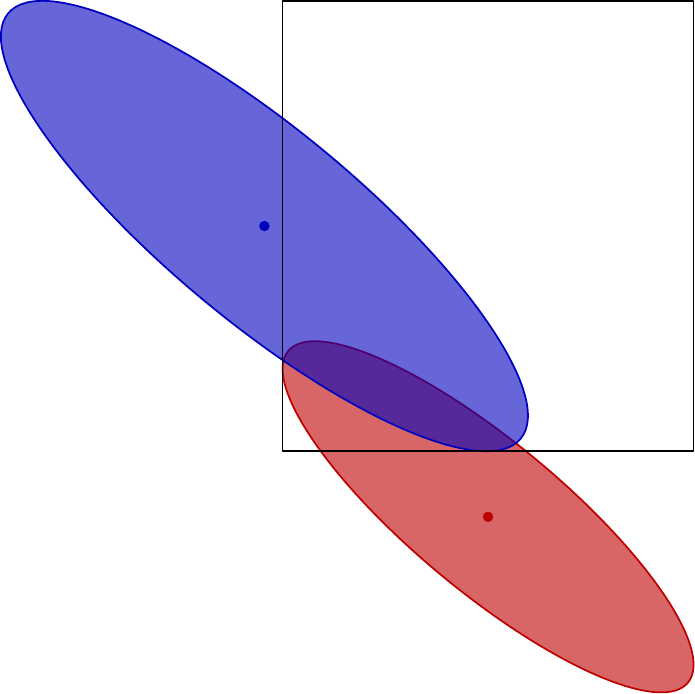}
    \caption{}
    \label{ellipse_static_2_volume_fraction}
  \end{subfigure}  
  \\[1cm]
  \begin{subfigure}{0.2\textwidth}
    \centering
    \includegraphics[scale=0.5]{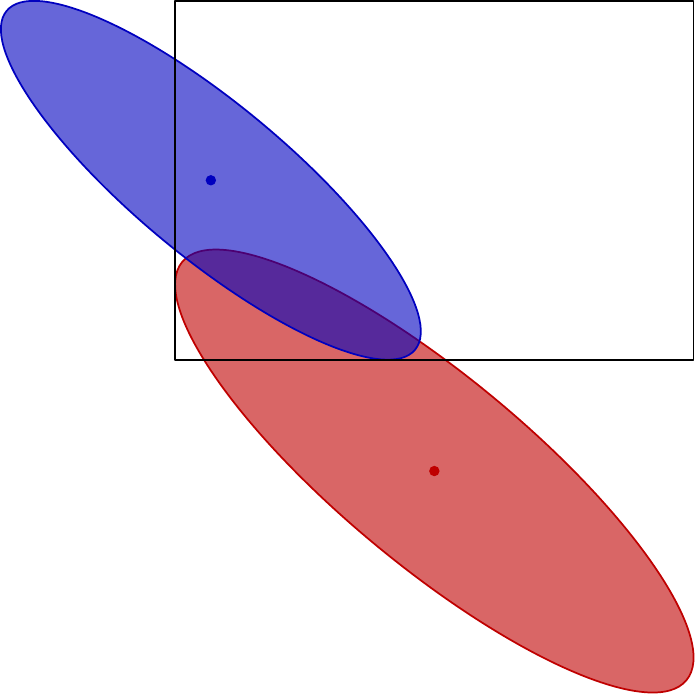}
    \caption{}
    \label{ellipse_static_2_p1c2p2c1}
  \end{subfigure}
  \qquad
  \begin{subfigure}{0.2\textwidth}
    \centering
    \includegraphics[scale=0.5]{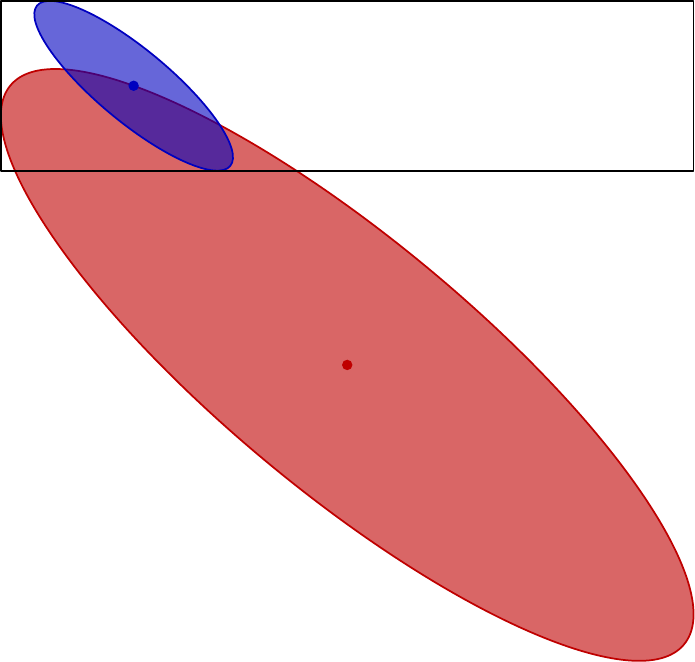}
    \caption{}
    \label{ellipse_static_2_p1c2zero}
  \end{subfigure}
  \qquad
  \begin{subfigure}{0.2\textwidth}
    \centering
    \includegraphics[scale=0.5]{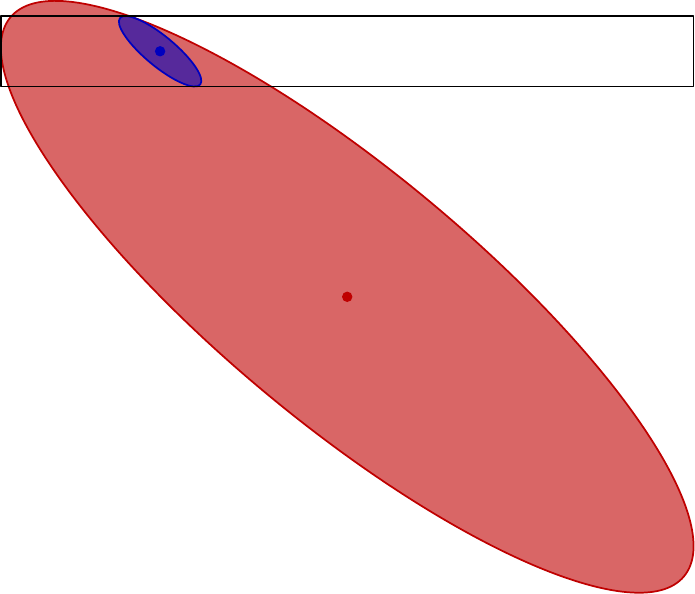}
    \caption{}
    \label{ellipse_static_2_seconddisczero}
  \end{subfigure}
  \qquad
  \begin{subfigure}{0.2\textwidth}
    \centering
    \includegraphics[scale=0.5]{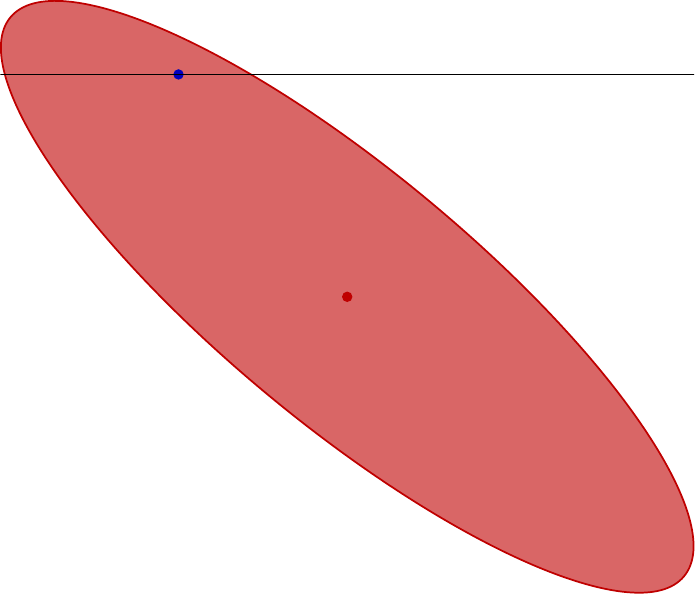}
    \caption{}
    \label{ellipse_static_2_elementary_upper_bound}
  \end{subfigure}  
  \\[1cm]
  \begin{subfigure}{\textwidth}
    \centering
    \includegraphics{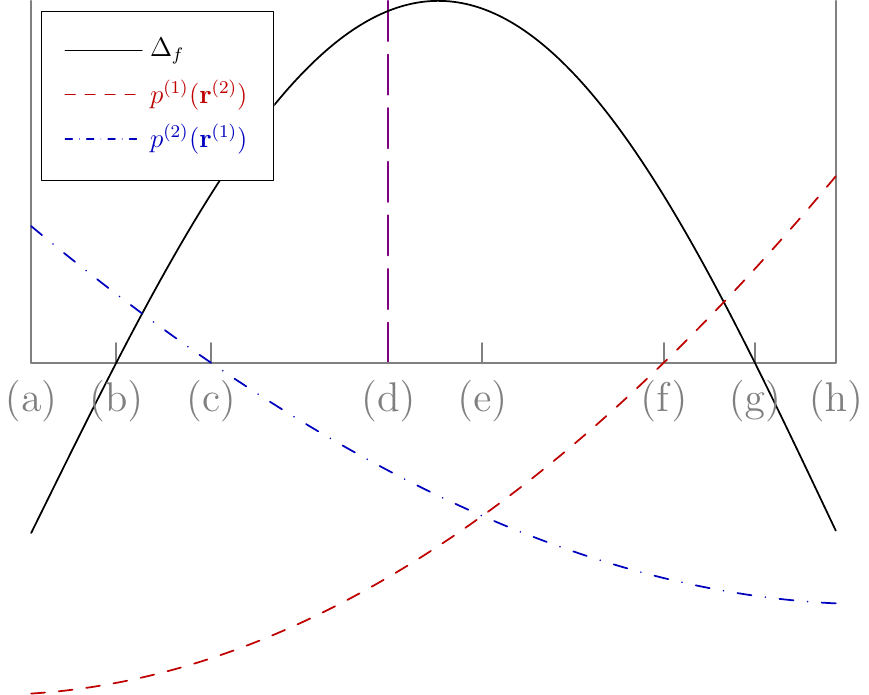}
    \caption{}
    \label{discriminant}
  \end{subfigure}  
  \caption{\emph{The rectangle $\F_{f,e}$ (outlined in black) and the sets $\cE_f^{(1)}$ (in red) and $\cE_f^{(2)}$ (in blue) at test values (a) $f=f_{e,l}\approx0.794$; (b) $f \approx 0.795$ (where $\Delta_f = 0)$; (c) $f \approx 0.797$ (where $p^{(2)}(\mathbf{r}^{(1)}) = 0$); (d) $f = f^{(1)} = 0.80$; (e) $f \approx 0.802$ (intersection of $p^{(1)}(\mathbf{r}^{(2)})$ and $p^{(2)}(\mathbf{r}^{(1)})$); (f) $f \approx 0.805$ (where $p^{(1)}(\mathbf{r}^{(2)}) = 0$); (g) $f \approx 0.806$ (where $\Delta_f = 0$); (h) $f = f_{e,u} \approx 0.808$.  (i) A plot of $\Delta_f$ (black solid line), $p_f^{(1)}(\mathbf{r}^{(2)})$ (red dashed line), and $p_f^{(2)}(\mathbf{r}^{(1)})$ (blue dash-dotted line) for $f \in \A_e = [f_{e,l},f_{e,u}]$ (the horizontal gray line is the $f-$axis).  The parameters used to create this figure are the same as those in Figure \ref{elementaryboundsfigure}.  In this case we only recover the elementary bounds $0.794 \le f^{(1)} \le 0.808$.}}  \label{ellipses_fig}
\end{figure}

In Figures \ref{ellipse_static_2_elementary_lower_bound}-\ref{ellipse_static_2_elementary_upper_bound} we plot the sets $\cE_f^{(1)}$ (red) and $\cE_f^{(2)}$ (blue) at various values of $f \in \A_e = [f_{e,l}, f_{e,u}]$; the centers of each ellipse are indicated by dots.  The black box is the boundary of the set $\F_{f,e}$, defined by the inequalities \eqref{xbounds} and \eqref{ybounds}.  Note that $\partial\cE_f^{(1)}$ is tangent to the vertical segments of the black box and $\partial\cE_f^{(2)}$ is tangent to the horizontal segments, as remarked after Lemma \ref{boundingbox}.  In particular, at $f = f_{e,l}$ (Figure \ref{ellipse_static_2_elementary_lower_bound}), $\cE_{f_{e,l}}^{(1)}$ is a point (represented by the red dot); at $f = f_{e,u}$ (Figure \ref{ellipse_static_2_elementary_upper_bound}), $\cE_{f_{e,u}}^{(2)}$ is a point (represented by the blue dot).  In Figure \ref{discriminant} we plot $\Delta_f$ (solid black line), $p_f^{(1)}(\mathbf{r}^{(2)})$ (red dashed line), and $p_f^{(2)}(\mathbf{r}^{(1)})$ (blue dash-dotted line) over the interval $\A_e$.  The true volume fraction is represented by the magenta dashed line and the horizontal gray line represents the $f-$axis.  Figure \ref{discriminant} shows that each $f \in \A_e$ is admissible; when $\Delta_f < 0$, we have either $p^{(2)}(\mathbf{r}^{(1)}) \ge 0$ and $p^{(1)}(\mathbf{r}^{(2)}) < 0$ (so $\cE^{(1)}_f \subset \cE^{(2)}_f$) or $p^{(1)}(\mathbf{r}^{(2)}) \ge 0$ and $p^{(2)}(\mathbf{r}^{(1)}) < 0$ (so $\cE^{(2)}_f \subset \cE^{(1)}_f$).  Thus for each $f \in \A_e$ the set $\F_f = \cE_f$ is nonempty and we conclude that $\A = \A_e$; in this example the bounds computed using the ellipses are no better than the elementary bounds.

\section{Additional Null Lagrangians in Two Dimensions}\label{section_additional_null_Lagrangians}

\subsection{Improved Elementary Bounds} \label{section_improved_elementary_bounds}
In two dimensions we can include information from the null Lagrangians $\lla\E_1\cdot R_{\perp} \E_2\rra$ and $\lla\J_1\cdot R_{\perp} \J_2\rra$--see \eqref{nullLagrangians2-D}.  The details presented below are similar in nature to those in the previous two sections.  

For arbitrary vectors $\mathbf{c}^{(\alpha)}, \mathbf{d}^{(\alpha)}$ in $\mathbb{R}^2$ and for $\alpha = 1, 2$ we define
\begin{equation}\label{h}
  \ha(\x;\mathbf{c}^{(\alpha)}, \mathbf{d}^{(\alpha)}) := \displaystyle\sum_{m=1}^{2} c^{(\alpha)}_m \left(\E^{(\alpha)}_m(\x) - \dfrac{\ca(\x)}{\fa}\langle\E_m^{(\alpha)}\rangle\right) + \displaystyle\sum_{n=1}^{2} d^{(\alpha)}_n \left(R_{\perp}\E^{(\alpha)}_n(\x) - \dfrac{\ca(\x)}{\fa}\langle R_{\perp} \E^{(\alpha)}_n\rangle\right) .
\end{equation}
Note that $\langle\ha\rangle = 0$ and $\langle\ha\cdot\ha\rangle \ge 0$ for all $\mathbf{c}^{(\alpha)}, \mathbf{d}^{(\alpha)} \in \mathbb{R}^2$.   A computation shows that this inequality is equivalent to 
\begin{equation}\label{2-Dstar}
  \mathbf{C}^{(\alpha)}\cdot \Ma \mathbf{C}^{(\alpha)} \ge 0 \text{ for all } \mathbf{C}^{(\alpha)} \in \mathbb{R}^4 ,
\end{equation}
where we have written 
\[ \mathbf{C}^{(\alpha)} = \begin{bmatrix} \mathbf{c}^{(\alpha)} \\ \mathbf{d}^{(\alpha)} \end{bmatrix}\]
for arbitrary $\mathbf{c}^{(\alpha)}, \mathbf{d}^{(\alpha)} \in \mathbb{R}^2$.  
For $\alpha = 1, 2$ the $4 \times 4$ matrix $\Ma$ is 
\begin{equation}\label{Malpha}
\Ma = \begin{bmatrix} \Sa & \Ta \\ - \Ta & \Sa \end{bmatrix} ,
\end{equation}
where
\begin{align}
\Ta &= \begin{bmatrix} \Ba_{11} - \dfrac{1}{\fa}\langle\E_1^{(\alpha)}\rangle \cdot \Rp \langle\E_1^{(\alpha)}\rangle & \Ba_{12} - \dfrac{1}{\fa}\langle\E_1^{(\alpha)}\rangle \cdot \Rp \langle\E_2^{(\alpha)}\rangle \\ \Ba_{21} - \dfrac{1}{\fa}\langle\E_2^{(\alpha)}\rangle \cdot \Rp \langle\E_1^{(\alpha)}\rangle & \Ba_{22} - \dfrac{1}{\fa}\langle\E_2^{(\alpha)}\rangle \cdot \Rp \langle\E_2^{(\alpha)}\rangle \end{bmatrix} , \\
\Ba_{mn} &= \langle \ca \E_m \cdot R_{\perp} \E_n \rangle = \langle \E_m^{(\alpha)}\cdot \Rp \E_n^{(\alpha)}\rangle \quad \text{ (for } m, n = 1, 2 \text{)} ,
\end{align}
and $R_{\perp}$ and $\Sa$ are as before (see \eqref{Rperp} and \eqref{Salpha}, respectively).  
Note that $\Ta_{11} = \Ta_{22} = 0$ for $\alpha = 1, 2$.  Also, $\Ta_{12} = -\Ta_{21}$ since $\Rp^T = -\Rp$.

For $f \in \A_e$ we define
\begin{equation}\label{finalMalpha}
  \Ma_f(x,y) := \begin{bmatrix} \Sa_f(x,y) & \Ta_f \\ -\Ta_f & \Sa_f(x,y) \end{bmatrix},
\end{equation}
where $\Sa_f(x,y)$ is defined in \eqref{Sxy},
\begin{equation*}
  \Ta_f = -\left[\Ta_f\right]^T = \begin{bmatrix} 0 & \sqrt{\tau_f^{(\alpha)}} \\ -\sqrt{\tau_f^{(\alpha)}} & 0 \end{bmatrix}
\end{equation*}
where 
\begin{equation}\label{tau}
  \tau^{(\alpha)}_f := \det \Ta_f = \left[\Ba_{12} - \dfrac{1}{f_*}\langle\E_1^{(\alpha)}\rangle \cdot \Rp \langle\E_2^{(\alpha)}\rangle\right]^2 \ge 0,
\end{equation}
and $f_*$ is defined in \eqref{fstar}.
Since $\Sa_f$ is symmetric for all $(x,y) \in \mathbb{R}^2$ and $\Ta_f$ is anti-symmetric, $\Ma_f(x,y)$ is symmetric for $f \in \A_e$ and all $(x,y) \in \mathbb{R}^2$. 

We apply the splitting method to $\lla \E_1 \cdot \Rp \E_2 \rra$ and $\lla \J_1 \cdot \Rp \J_2 \rra$ (see \eqref{splitting}) and obtain the system
\begin{equation}
  \begin{bmatrix} 1 & 1 \\ \left|\sigma^{(1)}\right|^2 & \left|\sigma^{(2)}\right|^2 \end{bmatrix} \begin{bmatrix} B^{(1)}_{12} \\[0.1cm] B^{(2)}_{12} \end{bmatrix} = \begin{bmatrix} \lla \E_1 \cdot \Rp \E_2 \rra \\[0.1cm] \lla \J_1 \cdot \Rp \J_2 \rra\end{bmatrix} .
\end{equation}
As long as $\left|\sigma^{(1)}\right| \ne \left|\sigma^{(2)}\right|$, we can solve this system for $B^{(1)}_{12}$ and $B^{(2)}_{12}$; in that case
\begin{equation}\label{B12}
  \begin{bmatrix} B^{(1)}_{12} \\[0.1cm] B^{(2)}_{12} \end{bmatrix} = \dfrac{1}{\left|\sigma^{(2)}\right|^2 - \left|\sigma^{(1)}\right|^2}\begin{bmatrix} \left|\sigma^{(2)}\right|^2 \langle \E_1 \cdot \Rp \E_2 \rangle -  \langle \J_1 \cdot \Rp \J_2 \rangle \\[0.1cm] -\left|\sigma^{(1)}\right|^2 \langle \E_1 \cdot \Rp \E_2 \rangle +  \langle \J_1 \cdot \Rp \J_2 \rangle\end{bmatrix} 
\end{equation}
and $B^{(1)}_{12}$ and $B^{(2)}_{12}$ (hence $T_f^{(1)}$ and $T_f^{(2)}$) are known.  

\begin{feasibleregion2}\label{feasibleregion2}
  For $f \in \A_e$ we set
  \begin{equation*}
    \tFa_f := \{(x,y) \in \mathbb{R}^2 : \Ma_f(x,y) \text{ is positive-semidefinite}\} .
  \end{equation*}
  Then the set $\tF_f := \tFo_f \cap \tFt_f$ is called the \emph{restricted feasible region associated with $f$}.  In addition, the set $\tA := \{f \in \A_e : \tF_f \ne \emptyset\}$ is called the \emph{restricted set of admissible test values}.  
\end{feasibleregion2}
To find the set $\tA$, we need to find the values of $f \in \A_e$ such that there is at least one point $(x,y) \in \mathbb{R}^2$ at which both $M^{(1)}_f(x,y)$ and $M^{(2)}_f(x,y)$ are simultaneously positive-semidefinite.  We will see that $\tA \subseteq \A$, so the bounds in this section are in general tighter than those in the previous sections.  

\begin{eigenvaluelemma}\label{eigenvaluelemma}
  Assume $\beta \ne 0$ and $\left|\sigma^{(1)}\right| \ne \left|\sigma^{(2)}\right|$.  Then for $f \in \A_e$ and $\alpha = 1, 2$, the matrix $\Ma_f(x,y)$ defined in \eqref{finalMalpha} is positive-semidefinite if and only if $\pa_f(x,y) = \det \Sa_f(x,y) \ge \tau^{(\alpha)}_f.$  
\end{eigenvaluelemma}
\begin{proof}
Recall that a symmetric matrix is positive-semidefinite if and only if all of its eigenvalues are nonnegative.  For $\alpha = 1, 2$ the eigenvalues  of $\Ma_f$, each with algebraic multiplicity 2, are
\begin{equation}\label{eigenvalues}
  \lambda^{(\alpha)}_{f,\pm}(x,y) = \frac{1}{2}\left\{\tr \Sa_f \pm \sqrt{\left[\tr\Sa_f\right]^2 - 4\left[\det\Sa_f-\det\Ta_f\right]}\right\}.
\end{equation}
(We have suppressed the dependence on $x$ and $y$ on the right-hand side of the above expression).   

By \eqref{Sxy}, \eqref{xs}, and \eqref{ys}, $\tr\Sa_f(x,y)$ is independent of $x$ and $y$ and is nonnegative if and only if $f \in \A_e$.  We note that the expression under the square root in \eqref{eigenvalues} must be nonnegative for all points $(x,y) \in \mathbb{R}^2$ and all $f \in \A_e$ since $M^{(\alpha)}_f(x,y)$ is symmetric for all such values of $x,y,$ and $f$.  

The previous paragraph implies that the eigenvalues $\lambda_{f,\pm}^{(\alpha)}(x,y)$ will be nonnegative for those points $(x,y) \in \mathbb{R}^2$ and those values of $f \in \A_e$ for which 
\begin{equation*}
  4\left[\det\Sa_f(x,y)-\det\Ta_f\right] \ge 0 \Leftrightarrow \det\Sa_f(x,y) \ge \tau^{(\alpha)}_f.
\end{equation*}
\end{proof}

Now $\pa_f \ge \tau^{(\alpha)}_f$ if and only if $\tpa_f \ge 0$, where $\tpa_f := \pa_f - \tau^{(\alpha)}_f$.  Using calculus we find
\begin{equation}\label{ptildemax}
  \tpa_{f,\max} := \displaystyle\max_{(x,y) \in \mathbb{R}^2} \tpa_f(x,y) 
  			  = \frac{1}{4 f_*^2}\left[\avgnormvplusa f_* - \normavgvplusa\right]\cdot\left[\avgnormvminusa f_* - \normavgvminusa \right]
\end{equation}
where $f_*$ is defined in \eqref{star} and 
\begin{equation}\label{vplusminusa}
  \vplusminusa := \ca\left(\E_1 \pm \Rp \E_2\right) = \E^{(\alpha)}_1 \pm \Rp \E^{(\alpha)}_2.
\end{equation}
Note that $\langle\vplusminusa\rangle = \langle\E^{(\alpha)}_1\rangle \pm \Rp \langle\E^{(\alpha)}_2\rangle$ is known (by the statement following \eqref{eq:E1E2}) and $\langle\vplusminusa\cdot\vplusminusa\rangle = \eta^{(\alpha)} \pm 2B_{12}^{(\alpha)}$ is known if and only if $\left|\sigma^{(1)}\right| \ne \left|\sigma^{(2)}\right|$ (by \eqref{constants} and \eqref{B12}).  For now we will assume that $\vplusminusa \not\equiv 0$ and $\eta^{(\alpha)} \ne 0$ (physically, this means that we assume that the real and imaginary parts of the electric field are nonperpendicular and nonzero in both phases).  We will show that $\tpa_{f,\max} < 0$ on a subset of $\A_e$; such values of $f$ are not admissible by Lemma \ref{eigenvaluelemma}.

Now $\widetilde{p}^{(1)}_{f,\max} \ge 0$ if and only if 
\begin{align}
  &f \ge \widetilde{f}_{e,l} := \max\left\{\frac{\normavgvminusone}{\avgnormvminusone},\frac{\normavgvplusone}{\avgnormvplusone}\right\}\label{tildefel} \\
  \text{or } &f \le Q^{(1)} := \min\left\{\frac{\normavgvminusone}{\avgnormvminusone},\frac{\normavgvplusone}{\avgnormvplusone}\right\} . \label{Qone}
\end{align}
A computation shows that $Q^{(1)} \le f_{e,l} \le \widetilde{f}_{e,l}$ and so the inequality in \eqref{Qone} will not be satisfied for all $f \ge f_{e,l}$ and can safely be ignored.  Moreover, we will have the chain of equalities $Q^{(1)} = f_{e,l} = \widetilde{f}_{e,l}$ if and only if 
\begin{equation}\label{eq:v1_condition}
  B_{12}^{(1)} f_{e,l} = \langle\Eoo\rangle\cdot\Rp\langle\Eto\rangle.
\end{equation}
If $\E^{(1)}$ is a constant, then \eqref{eq:v1_condition} becomes $f_{e,l} = f^{(1)}$, which is consistent with our work in Section \ref{section_elementary_bounds}.  We also note that $\widetilde{f}_{e,l}$ can be rewritten as
\begin{equation}\label{eq:alternative_aelb}
  \widetilde{f}_{e,l} = \begin{cases}  \dfrac{\normavgvplusone}{\avgnormvplusone} & \text{if } B_{12}^{(1)} f_{e,l} \le \langle\Eoo\rangle\cdot\Rp\langle\Eto\rangle \\[0.5cm] \dfrac{\normavgvminusone}{\avgnormvminusone} & \text{if } B_{12}^{(1)} f_{e,l} > \langle\Eoo\rangle\cdot\Rp\langle\Eto\rangle. \end{cases} 
\end{equation}
The above computations are summarized in Figure \ref{parabolas}, which is a plot of the functions $\tpa_{f,\max}$ as a function of $f$.  The function $\widetilde{p}^{(1)}_{f,\max}$ is plotted as a red solid curve.  If \eqref{eq:v1_condition} does not hold, its zeros $Q^{(1)}$ and $\widetilde{f}_{e,l}$ are below and above the elementary lower bound $f_{e,l}$, respectively.  Thus all values of $f \in [f_{e,l},\widetilde{f}_{e,l})$ are not admissible, giving us the improved elementary lower bound $\widetilde{f}_{e,l} \le \fone$.  If \eqref{eq:v1_condition} holds, then $Q^{(1)} = \widetilde{f}_{e,l} = f_{e,l}$, and we do not obtain an improved elementary lower bound.  In Figure \ref{parabolas}, $f_{e,l}$ is indicated with the left gray vertical line while $\widetilde{f}_{e,l}$ is indicated by the left black vertical line.

Similarly, $\widetilde{p}^{(2)}_{f,\max} \ge 0$ if and only if
\begin{align}
  &f \le \widetilde{f}_{e,u}:= \min\left\{1-\frac{\normavgvminustwo}{\avgnormvminustwo},1-\frac{\normavgvplustwo}{\avgnormvplustwo}\right\} \label{tildefeu} \\
  \text{or } &f \ge Q^{(2)}:= \max\left\{1-\frac{\normavgvminustwo}{\avgnormvminustwo},1-\frac{\normavgvplustwo}{\avgnormvplustwo}\right\}. \label{Qtwo}
\end{align}
Again a computation shows that $\widetilde{f}_{e,u} \le f_{e,u} \le Q^{(2)}$.  We will have the chain of equalities $\widetilde{f}_{e,u} = f_{e,u} = Q^{(2)}$ if and only if 
\begin{equation}\label{eq:v2_condition}
B_{12}^{(2)}(1-f_{e,u}) = \langle\Eot\rangle\cdot\Rp\langle\Ett\rangle.
\end{equation}  
If $\E^{(2)}$ is a constant, then \eqref{eq:v2_condition} becomes $f_{e,u} = f^{(1)}$, which is consistent with our work in Section \ref{section_elementary_bounds}.  We also note that $\widetilde{f}_{e,u}$ can be rewritten as
\begin{equation}\label{eq:alternative_aeub}
  \widetilde{f}_{e,u} = \begin{cases}  1-\dfrac{\normavgvplustwo}{\avgnormvplustwo} & \text{if } B_{12}^{(2)}(1-f_{e,u}) \le \langle\Eot\rangle\cdot\Rp\langle\Ett\rangle \\[0.5cm] 1-\dfrac{\normavgvminustwo}{\avgnormvminustwo} & \text{if } B_{12}^{(2)}(1-f_{e,u}) > \langle\Eot\rangle\cdot\Rp\langle\Ett\rangle. \end{cases} 
\end{equation}
The function $\widetilde{p}^{(2)}_{f,\max}$ is plotted as a blue dashed curve in Figure \ref{parabolas}.  If \eqref{eq:v2_condition} does not hold the values of $f \in (\widetilde{f}_{e,u},f_{e,u}]$ are not admissible so we obtain the improved elementary upper bound $f \le \widetilde{f}_{e,u}$; if \eqref{eq:v2_condition} holds then $Q^{(2)} = \widetilde{f}_{e,u} = f_{e,u}$ and we do not obtain an improved elementary upper bound.  In Figure \ref{parabolas}, $f_{e,u}$ and $\widetilde{f}_{e,u}$ are indicated by the right gray and black vertical lines, respectively.  

Finally, we can show that $\widetilde{f}_{e,l} \le \widetilde{f}_{e,u}$ and provide a much simpler derivation of the improved elementary bounds as follows.  We begin by noting that
\begin{equation}
  \lla\left\|\vplusminusa - \frac{\ca}{\fa}\langle\vplusminusa\rangle\right\|^2\rra \ge 0, 
\end{equation}
which is equivalent to 
\begin{equation}\label{Caucyv}
  \normavgvplusminusa \le \fa\avgnormvplusminusa \Leftrightarrow \frac{\normavgvplusminusa}{\avgnormvplusminusa} \le \fa
\end{equation}
with equality if and only if $\vplusminusa$ is a (nonzero) constant.  This implies that 
\begin{equation*}
  \widetilde{f}_{e,l} \le f^{(1)} \quad \text{ and } \quad  \widetilde{f}_{e,u} \ge 1-f^{(2)} = f^{(1)}.
\end{equation*}
The first inequality above will be satisfied as an equality if and only if $\vplusone$ or $\vminusone$ is a (nonzero) constant; the second inequality above will be satisfied as an equality if and only if $\vplustwo$ or $\vminustwo$ is a (nonzero) constant.
\begin{restricted_elementary_admissible}\label{restricted_elementary_admissible}
  The set $\widetilde{\A}_e := \{f \in \A_e : \widetilde{f}_{e,l} \le f \le \widetilde{f}_{e,u} \}$ is called the \emph{restricted elementary set of admissible test values}.  
\end{restricted_elementary_admissible}
The set $\widetilde{\A}_e$ is highlighted by the darkened interval in Figure \ref{parabolas}, while the true volume fraction $f^{(1)}$ is indicated by the magenta vertical dashed line.  We note that $\widetilde{\A}_e \subseteq \A_e$ with equality if and only if \eqref{eq:v1_condition} and \eqref{eq:v2_condition} hold.
\begin{figure}[!hbtp]
  \centering
  \begin{subfigure}{0.45\textwidth}
    \centering
    \includegraphics[scale=0.7]{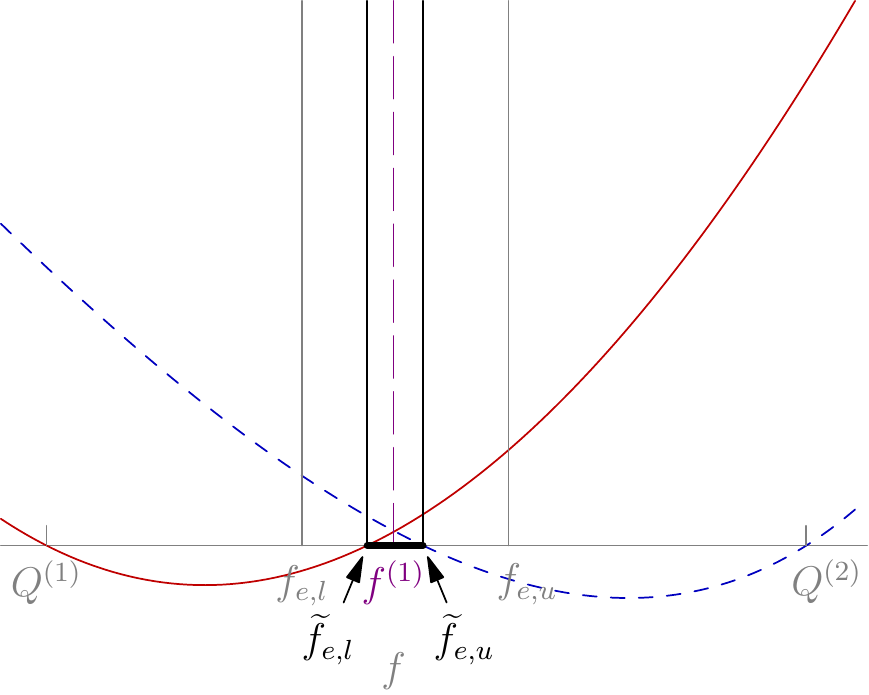}
    \caption{}
    \label{quadratics_subfigure}
  \end{subfigure}
  \qquad
  \begin{subfigure}{0.45\textwidth}
    \centering
    \includegraphics[scale=0.7]{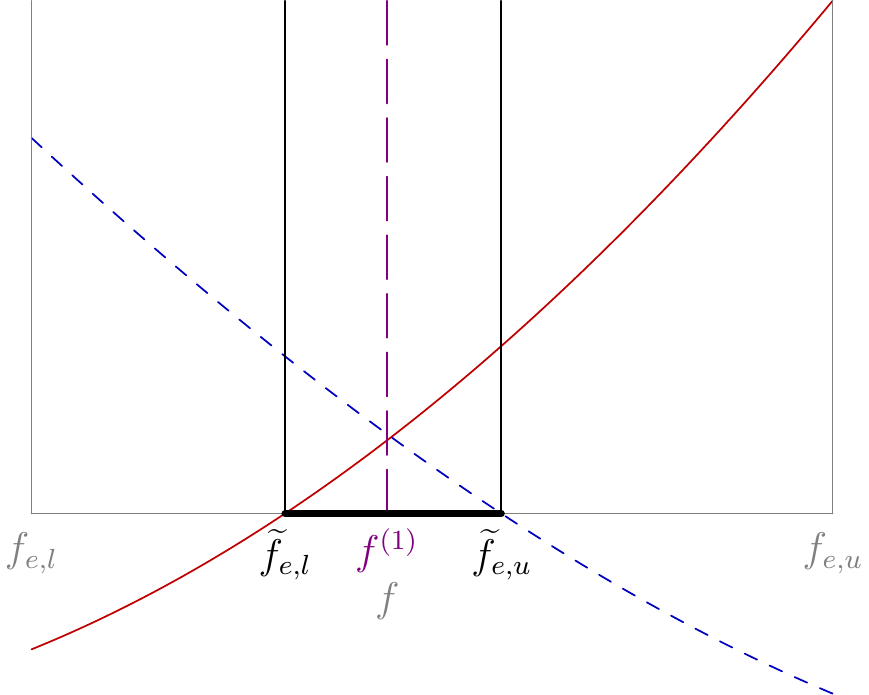}
    \caption{}
    \label{quadratics_zoom}
  \end{subfigure}
  \caption{\emph{(a) A plot of $\widetilde{p}^{(1)}_{f,\max}$ (red solid curve) and $\widetilde{p}^{(2)}_{f,\max}$ (blue dashed curve)--the horizontal gray line represents the $f-$axis.  The geometry and parameters used to create these plots are the same as those used to create Figure \ref{elementaryboundsfigure}.  (b) A zoomed-in version of (a)--here we plot the functions over the interval $[f_{e,l},f_{e,u}]$. In both figures the set $\widetilde{A}_e = [\widetilde{f}_{e,l},\widetilde{f}_{e,u}]$ is highlighted by the darkened interval.  Some relevant numbers are $f_{e,l} \approx 0.794, f_{e,u} \approx 0.808, \widetilde{f}_{e,l} \approx 0.798, \widetilde{f}_{e,u} \approx 0.802, Q^{(1)} \approx 0.776, Q^{(2)} \approx 0.828,$ and $f^{(1)} = 0.8$.  So we obtain the better bounds $0.798 \le f^{(1)} \le 0.802$.}}
  \label{parabolas}
\end{figure}

We have thus proven the following theorem.
\begin{improved_elementary_bounds_theorem}[Improved Elementary Bounds]\label{improved_elementary_bounds_theorem}
  Suppose $\beta \ne 0$, $\left|\sigma^{(1)}\right| \ne \left|\sigma^{(2)}\right|$, $\mathbf{v}_{\pm}^{(\alpha)} \not\equiv 0$, and $\eta^{(\alpha)} \ne 0$ for $\alpha = 1, 2$.  Then the volume fraction $\fone = \langle\chione\rangle$ satisfies the bounds $\widetilde{f}_{e,l} \le \fone \le \widetilde{f}_{e,u}$ where $\widetilde{f}_{e,l}$ and $\widetilde{f}_{e,u}$ are defined in \eqref{tildefel} and \eqref{tildefeu}, respectively (also see \eqref{eq:alternative_aelb} and \eqref{eq:alternative_aeub}).  
Moreover, the lower bound is satisfied as an equality (i.e. $\widetilde{f}_{e,l} = \fone$) if and only if $\mathbf{v}_+^{(1)}$ or $\mathbf{v}_-^{(1)}$ is a nonzero constant while the upper bound is satisfied as an equality (i.e. $\widetilde{f}_{e,u} = \fone$) if and only if $\mathbf{v}_+^{(2)}$ or $\mathbf{v}_-^{(2)}$ is a nonzero constant.
Finally, these are tighter bounds than those discussed in Theorem \ref{elementary_bounds_theorem}, i.e. $f_{e,l} \le \widetilde{f}_{e,l}$ with equality if and only if \eqref{eq:v1_condition} holds and $\widetilde{f}_{e,u} \le f_{e,u}$ with equality if and only if \eqref{eq:v2_condition} holds.
\end{improved_elementary_bounds_theorem}

\subsection{Example of the Improved Elementary Bounds Being Attained}\label{subsection:Attainability_Conditions}
  We now consider a configuration of concentric disks for which the improved elementary lower bound from Section \ref{section_improved_elementary_bounds} gives the exact volume fraction while the original elementary lower bound from Section \ref{section_elementary_bounds} only gives a lower bound on the volume fraction.  Thus for this example we will see that
  \[f_{e,l} < \widetilde{f}_{e,l} = f^{(1)} < \widetilde{f}_{e,u} < f_{e,u}.\]
  We denote the radii and conductivities of the inner disk (core) and outer annulus (shell) by $R_1$ and $R_2$ and $\sigma^{(1)}$ and $\sigma^{(2)}$, respectively.  Throughout this section we will take $z = x+\ii y = r\ee^{\ii\theta}$; the complex conjugate of $z$ will be denoted by $\overline{z}$ and is given by $\overline{z} = x-\ii y = r\ee^{-\ii\theta}$.  We note that the condition $\vplusa$ being constant is equivalent to the potential in phase $\alpha$ being the sum of function linear in $z$ plus a function $g(\overline{z})$ or conversely $\vminusa$ being constant is equivalent to the potential in phase $\alpha$ being a function linear in $\overline{z}$ plus a function $g(z)$.

We will take the Dirichlet boundary condition 
\begin{equation}
  V(R_2,\theta) = V_0(R_2,\theta) = \left(a R_2 + \frac{b}{R_2}\right)\ee^{\ii\theta} + \left(c R_2^2 + \frac{d}{R_2^2}\right)\ee^{-2\ii\theta},
\end{equation}
where
\begin{equation}\label{eq:core_shell_coefficients}
  a = \frac{\sigma^{(1)}+\sigma^{(2)}}{2\sigma^{(2)}}; \quad b = -\frac{R_1^2[\sigma^{(1)}+\sigma^{(2)}]}{2\sigma^{(2)}}; \quad c = \frac{k[\sigma^{(1)}+\sigma^{(2)}]}{2\sigma^{(2)}}; \quad d = -\frac{kR_1^4[\sigma^{(1)}-\sigma^{(2)}]}{2\sigma^{(2)}}.
\end{equation}
and $k \in \mathbb{R}$ (entering \eqref{eq:core_shell_coefficients}) is a given constant.
The potential in the core (for $0 < r < R_1$) is then given by 
\begin{equation}\label{eq:core_potential}
  V^{(1)}(z,\overline{z}) = z + k(\overline{z})^2.
\end{equation}
The potential in the shell ($R_1 < r < R_2$) can be found by using the continuity of the potential $V$ and the current $-\sigma \nabla V \cdot \mathbf{n}$ across the boundary at $r = R_1$; in particular we find 
\begin{equation}\label{eq:shell_potential}
  V^{(2)}(z,\overline{z}) = a z + \frac{b}{\overline{z}} + c(\overline{z})^2 + \frac{d}{z^2},
\end{equation}
where $a, b, c,$ and $d$ are given in \eqref{eq:core_shell_coefficients}.
Let 
\[ \xhat = \begin{bmatrix} 1 \\ 0 \end{bmatrix} \quad \text{and} \quad \yhat = \begin{bmatrix} 0 \\ 1 \end{bmatrix} \]
be the standard orthonormal basis for $\mathbb{R}^2$.  Then, since $\E = -\nabla V$, the electric field in each phase is given by 
\begin{equation}\label{eq:electric_field_core_shell}
  \begin{aligned}
    \E^{(1)} &= -\left(1+2k\overline{z}\right)\xhat - \ii\left(1-2 k \overline{z}\right)\yhat \\
    \E^{(2)} &= -\left[a-\frac{b}{(\overline{z})^2}+2c\overline{z} - \frac{2-D}{z^3}\right]\xhat - \ii\left[a+\frac{b}{(\overline{z})^2}-2c\overline{z}-\frac{2-D}{z^3}\right]\yhat.
  \end{aligned}
\end{equation}
We emphasize that neither of these fields is constant; therefore Theorem \ref{elementary_bounds_theorem} implies 
\[f_{e,l} < f^{(1)} < f_{e,u}.\]
In particular
\begin{equation}\label{eq:core_elb}
  f_{e,l} = \left(\frac{1}{1+2k^2R_1^2}\right)\frac{R_1^2}{R_2^2}.
\end{equation}
For $k \ne 0$ this is strictly less than $f^{(1)} = \dfrac{R_1^2}{R_2^2}$.

Recall that $\vplusminusa = \E^{(\alpha)}_1 \pm \Rp \E^{(\alpha)}_2$.  We can compute
\begin{equation}\label{eq:v1pm_core}
  \vplusone = -2\xhat \quad \text{and} \quad \vminusone = 4k\left(-x\xhat + y\yhat\right).
\end{equation}
So $\vplusone$ is a constant.  We note that both fields $\vplusminustwo$ are not uniform.  Theorem \ref{improved_elementary_bounds_theorem} thus implies that $\widetilde{f}_{e,l} = f^{(1)}$ and $f^{(1)} < \widetilde{f}_{e,u}$.

Finally, if $k = 0$ note that \eqref{eq:electric_field_core_shell} implies that $\E^{(1)}= -\xhat -\ii\yhat$ is a constant.  Thus Theorem \ref{elementary_bounds_theorem} implies that $f_{e,l} = f^{(1)}$, which is verified by \eqref{eq:core_elb}.  Additionally  \eqref{eq:v1pm_core} implies that $\vminusone \equiv 0$, so Theorem \ref{improved_elementary_bounds_theorem} implies that $\widetilde{f}_{e,l} = f_{e,l}$.

\subsection{More Sophisticated Bounds}\label{section_tilde_ellipse_bounds}
We now proceed to find improved bounds; the method is very similar to that in Section \ref{section_more_sophisticated_bounds}.
\begin{ellipseregion2}\label{ellipseregion2}
  For $\alpha = 1, 2$ and for $f \in \tA_e$ we define
  \begin{equation*}
    \tEa_f := \{(x,y) \in \mathbb{R}^2 : \pa_f(x,y) \ge \tau^{(\alpha)}_f\} \quad \text{and} \quad \tcE_f := \tEo_f\cap\tEt_f .
  \end{equation*}
\end{ellipseregion2}
Since $\tau^{(\alpha)}_f \ge 0$, Lemma \ref{eigenvaluelemma} implies that $\tEa_f \subseteq \Ea_f$; that is, the elliptic disks in this case will be smaller than those in the previous section (for which $\tau^{(\alpha)}_f \equiv 0$).  For each $f \in \widetilde{\A}_e$ we check to see whether or not $\tcE_f$ is empty.  If $\tcE_f \ne \emptyset$ then $f \in \tA$; if $\tcE_f = \emptyset$ then $f \notin \tA$.  As in Section \ref{section_more_sophisticated_bounds}, we cannot work through everything explicitly due to the complexity of the expressions involved.  However, Lemmas \ref{ellipseslemma}-\ref{twointersectionpoints} (and therefore Theorem \ref{noRperptheorem}) extend immediately; we present their extensions here for completeness.  

\begin{ellipseslemma2}\label{ellipseslemma2}
  Assume that $\beta \ne 0$, $\left|\sigma^{(1)}\right| \ne \left|\sigma^{(2)}\right|$, $\vplusminusa \ne 0$, and $\eta^{(\alpha)} \ne 0$ for $\alpha = 1, 2$.  Then the following properties hold.
  \begin{enumerate}
    \item[(1)] For $f \in (\widetilde{f}_{e,l}, \widetilde{f}_{e,u})$ and $\alpha = 1, 2$, $\tEa_f$ is a closed elliptic disk; its boundary is the ellipse $\partial\tEa_f = \{(x, y) \in \mathbb{R}^2 : \tpa_f(x,y) = 0\};$
    \item[(2)] $\tcE^{(1)}_{\widetilde{f}_{e,l}}$ is a point and $\tcE^{(2)}_{\widetilde{f}_{e,l}}$ is a closed elliptic disk;
    \item[(3)] $\tcE^{(1)}_{\widetilde{f}_{e,u}}$ is a closed elliptic disk and $\tcE^{(2)}_{\widetilde{f}_{e,u}}$ is a point.
  \end{enumerate}
\end{ellipseslemma2}
\begin{proof}
  We simply apply the proof of Lemma \ref{ellipseslemma} to $\tpa_f$.
\end{proof}
\begin{boundingbox2}\label{boundingbox2}
  Suppose $\beta \ne 0$, $\left|\sigma^{(1)}\right| \ne \left|\sigma^{(2)}\right|$, $\vplusminusa \ne 0$, and $\eta^{(\alpha)} \ne 0$ for $\alpha = 1, 2$.  Then for each $f \in \widetilde{\A}_e, \tcE_f \subseteq \F_{f,e}$.
\end{boundingbox2}
\begin{proof}
 For each $f \in \widetilde{\A}_e$, $\tcE_f \subseteq \cE_f$ by Lemma \ref{eigenvaluelemma}; since $\cE_f \subseteq \F_{f,e}$ for each $f \in \A_e \supseteq \tA_e$ by Lemma \ref{boundingbox}, $\tcE_f \subseteq \F_{f,e}$ for each $f \in \tA_e$.
\end{proof}
\begin{twointersectionpoints2}\label{twointersectionpoints2}
  Suppose $\beta \ne 0$, $\left|\sigma^{(1)}\right| \ne \left|\sigma^{(2)}\right|$, $\vplusminusa \ne 0$, and $\eta^{(\alpha)} \ne 0$ for $\alpha = 1, 2$.  Then for each $f \in \widetilde{\A}_e$ the set $\partial \tcE^{(1)}_f \cap \partial \tcE^{(2)}_f$ contains at most 2 points.
\end{twointersectionpoints2}
\begin{proof}
  The proof is a word-for-word repeat of the proof of Lemma \ref{twointersectionpoints} applied to $\tpa_f$.
\end{proof}
Therefore we can numerically search for tighter bounds as follows.  For each $f \in \widetilde{\A}_e$, if $\widetilde{\Delta}_f \ge 0$ then $f \in \tA$ (where $\widetilde{\Delta}_f$ is the same as $\Delta_f$ but with $\al_6$ replaced by $\widetilde{a}^{(\alpha)}_6 := \al_6 - \tau^{(\alpha)}_f$).  If $\widetilde{\Delta}_f < 0$, then $f \notin \tA$ if and only if $\widetilde{p}^{(1)}_f(\mathbf{r}^{(2)}) < 0 $ and $\widetilde{p}^{(2)}(\mathbf{r}^{(1)}) < 0$, where $\mathbf{r}^{(1)}$ and $\mathbf{r}^{(2)}$ are defined in \eqref{center}.
We have thus proven the following theorem.
\begin{Rperptheorem}\label{Rperptheorem}
  Suppose $\beta \ne 0$, $\left|\sigma^{(1)}\right| \ne \left|\sigma^{(2)}\right|$, $\vplusminusa \ne 0$, and $\eta^{(\alpha)} \ne 0$ for $\alpha = 1, 2$. Then for $f \in \widetilde{\A}_e \, (= [\widetilde{f}_{e,l}, \widetilde{f}_{e,u}])$, if $\widetilde{\Delta}_f \ge 0$ then $f \in \tA$ where $\widetilde{\Delta}_f$ is defined in \eqref{Delta} by replacing $\al_6$ by $\widetilde{a}^{(\alpha)}_6 = \al_6 - \tau^{(\alpha)}_f$.  If $\widetilde{\Delta}_f < 0$, then $f \notin \tA$ if and only if $\widetilde{p}^{(1)}_f(\mathbf{r}^{(2)}) < 0$ and $\widetilde{p}^{(2)}_f\left(\mathbf{r}^{(1)}\right) < 0$.
\end{Rperptheorem}
The numerically computed bounds may or may not be tighter than the improved elementary bounds, depending on the problem under consideration--see the last paragraph in Section \ref{section_elementary_bounds}.  If we consider concentric disks in which the inner disk is labeled as phase 1, then the improved elementary lower bound will be exactly equal to the volume fraction, i.e. $\widetilde{f}_{e,l} = \fone$.  In this case the field inside the inner disk is constant, so $\vplusone$ and $\vminusone$ are both constants as well.  This example is somewhat trivial in the sense that the original elementary lower bound is also equal to the volume fraction, i.e. $f_{e,l} = \fone$ (see the last paragraph in Section \ref{section_elementary_bounds}).  In the case of a two-phase simple laminate we find that $f_{e,l} = \widetilde{f}_{e,l} = \widetilde{f}_{e,u} = f_{e,u} = \fone$ since the electric field is constant in both phases.  In Section \ref{subsection:Attainability_Conditions} we gave an example of a geometry in which the improved elementary lower bound $\widetilde{f}_{e,l}$ is equal to the true volume fraction $f^{(1)}$ but the elementary lower bound $f_{e,l}$ is strictly less than the volume fraction. 

In Figures \ref{tilde_ellipse_static_2_improved_elementary_lower_bound}-\ref{tilde_ellipse_static_2_improved_elementary_upper_bound} we plot the sets $\tcE_f^{(1)}$ (red) and $\tcE_f^{(2)}$ (blue) at various values of $f \in \tA_e = [\widetilde{f}_{e,l}, \widetilde{f}_{e,u}]$; the centers of each ellipse are indicated by a dot while the black box is the boundary of the set $\F_{f,e}$ (see Definition \ref{elementaryfeasibleregion}).  For comparison we plot $\cE_f^{(1)}$ (red dashed ellipse) and $\cE_f^{(2)}$ (blue dashed ellipse).  Note that $\cE_f\ne\emptyset$ in Figures \ref{tilde_ellipse_static_2_improved_elementary_lower_bound}-\ref{tilde_ellipse_static_2_improved_elementary_upper_bound} but $\tcE_f \ne \emptyset$ only in Figures \ref{tilde_ellipse_static_2_firsttildedisczero}-\ref{tilde_ellipse_static_2_secondtildedisczero}.  
In Figure \ref{discriminant_tilde} we plot $\widetilde{\Delta}_f$ (solid black line), $\widetilde{p}_f^{(1)}(\mathbf{r}^{(2)})$ (red dashed line), and $\widetilde{p}_f^{(2)}(\mathbf{r}^{(1)})$ (blue dash-dotted line) over the interval $\tA_e$.  The true volume fraction is represented by the magenta dashed line and the horizontal gray line represents the $f-$axis.  In addition, the set $\tA$ is indicated by the darkened interval.  In this case $\tA \subset \tA_e$ (which is in contrast to the example in Figure \ref{ellipses_fig} where $\A = \A_e$)--since $\widetilde{p}^{(1)}(\mathbf{r}^{(2)})$ and $\widetilde{p}^{(2)}(\mathbf{r}^{(1)})$ are both negative for all $f \in \tA_e$, the set $\tA$ is simply the set on which $\widetilde{\Delta}_f \ge 0$.
\begin{figure}[!hbtp]
  \centering
  \begin{subfigure}{0.2\textwidth}
    \centering
    \includegraphics[scale=0.5]{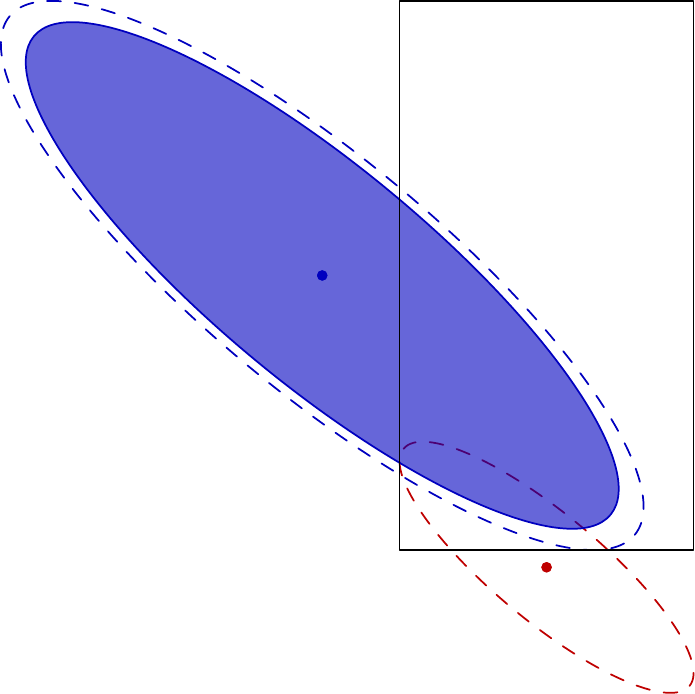}
    \caption{}
    \label{tilde_ellipse_static_2_improved_elementary_lower_bound}
  \end{subfigure}
  \qquad
  \begin{subfigure}{0.2\textwidth}
    \centering
    \includegraphics[scale=0.5]{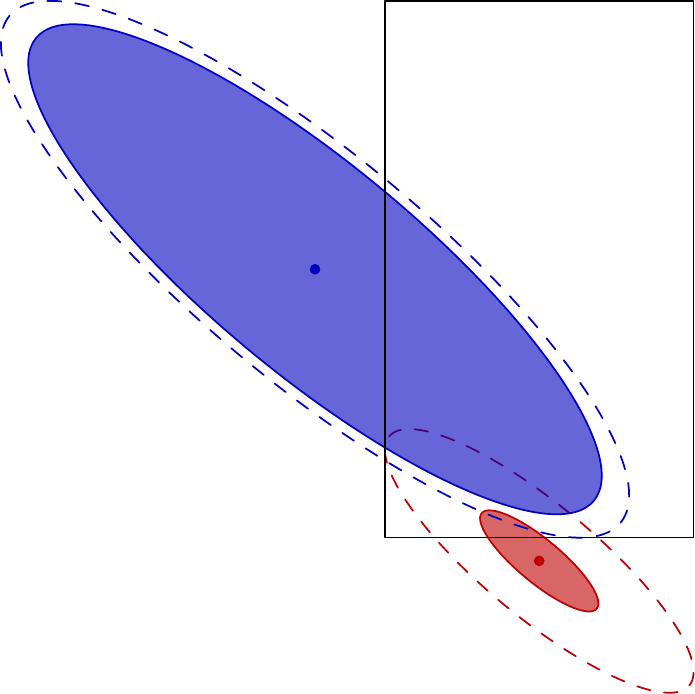}
    \caption{}
    \label{tilde_ellipse_static_2_average_improved_elementary_lower_bound_firsttildedisczero}
  \end{subfigure}
  \qquad
  \begin{subfigure}{0.2\textwidth}
    \centering
    \includegraphics[scale=0.5]{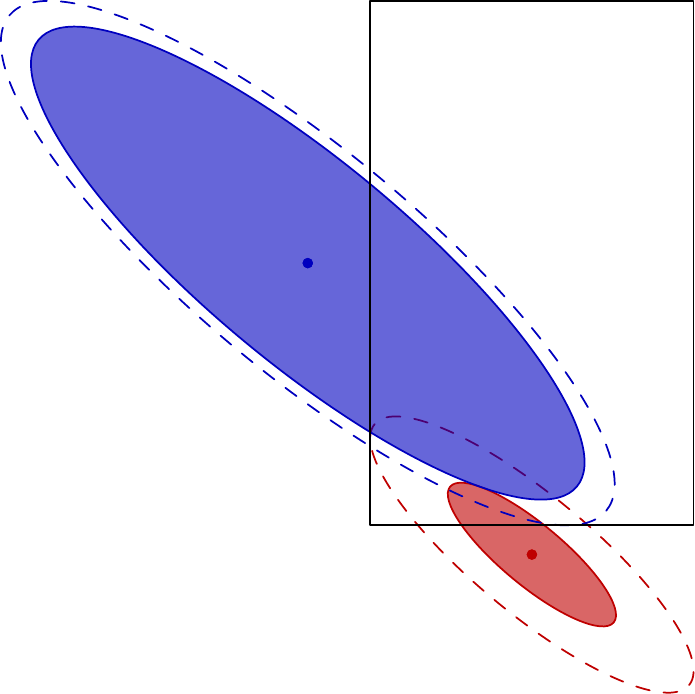}
    \caption{}
    \label{tilde_ellipse_static_2_firsttildedisczero}
  \end{subfigure}
  \qquad
  \begin{subfigure}{0.2\textwidth}
    \centering
    \includegraphics[scale=0.5]{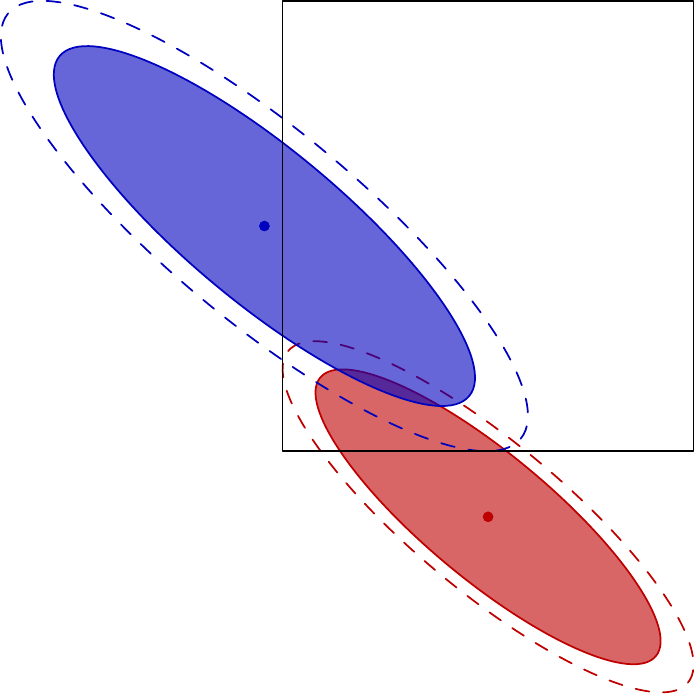}
    \caption{}
    \label{tilde_ellipse_static_2_volume_fraction}
  \end{subfigure}  
  \\[1cm]
  \begin{subfigure}{0.2\textwidth}
    \centering
    \includegraphics[scale=0.5]{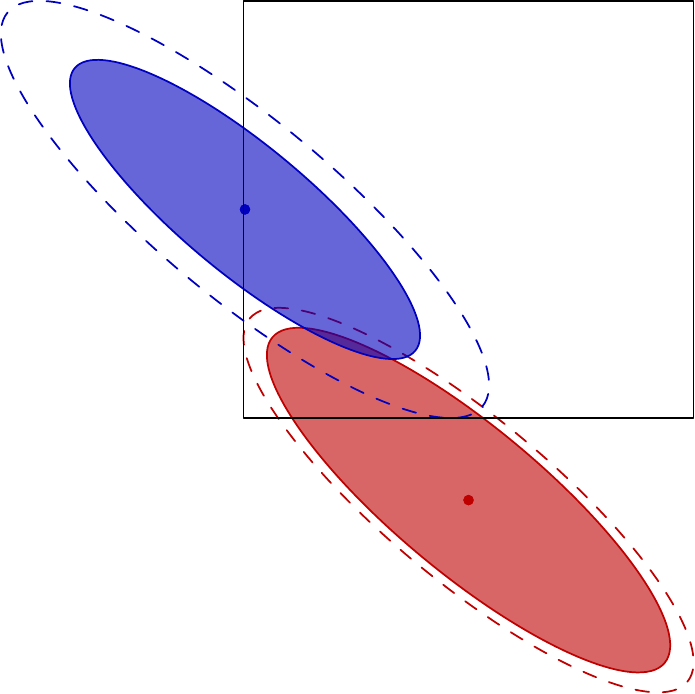}
    \caption{}
    \label{tilde_ellipse_static_2_average_volume_fraction_secondtildedisczero}
  \end{subfigure}
  \qquad
  \begin{subfigure}{0.2\textwidth}
    \centering
    \includegraphics[scale=0.5]{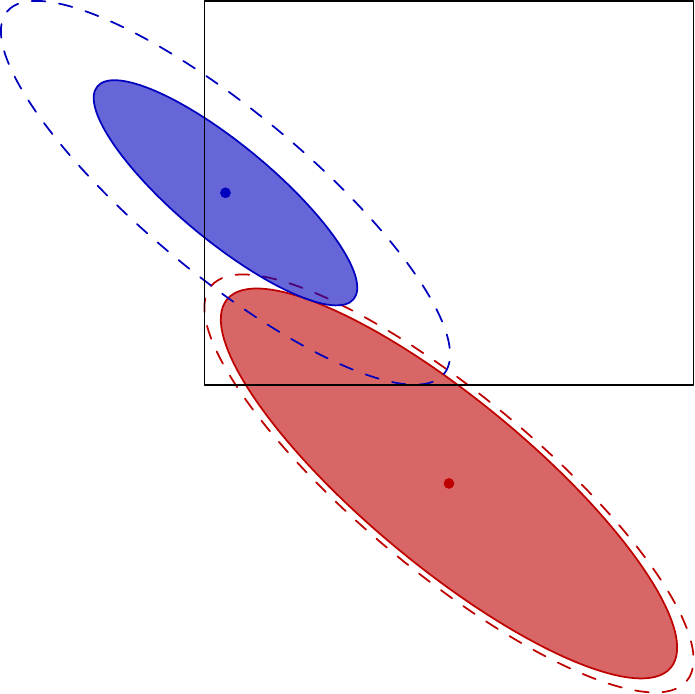}
    \caption{}
    \label{tilde_ellipse_static_2_secondtildedisczero}
  \end{subfigure}
  \qquad
  \begin{subfigure}{0.2\textwidth}
    \centering
    \includegraphics[scale=0.5]{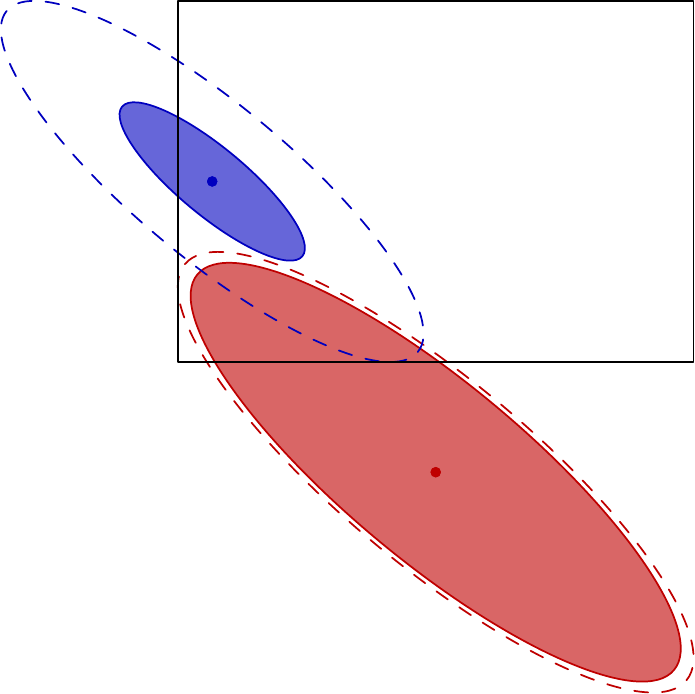}
    \caption{}
    \label{tilde_ellipse_static_2_averaga_improved_elementary_upper_bound_secondtildedisczero}
  \end{subfigure}
  \qquad
  \begin{subfigure}{0.2\textwidth}
    \centering
    \includegraphics[scale=0.5]{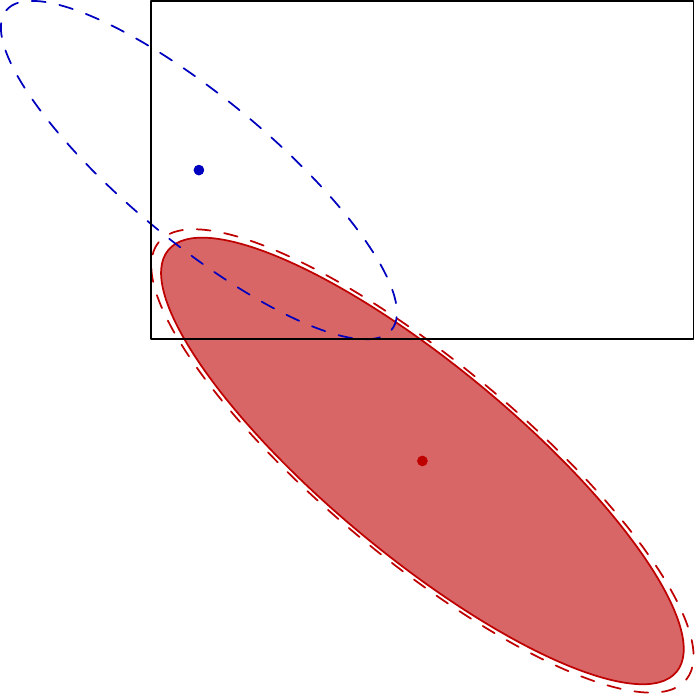}
    \caption{}
    \label{tilde_ellipse_static_2_improved_elementary_upper_bound}
  \end{subfigure}  
  \\[1cm]
  \begin{subfigure}{\textwidth}
    \centering
    \includegraphics{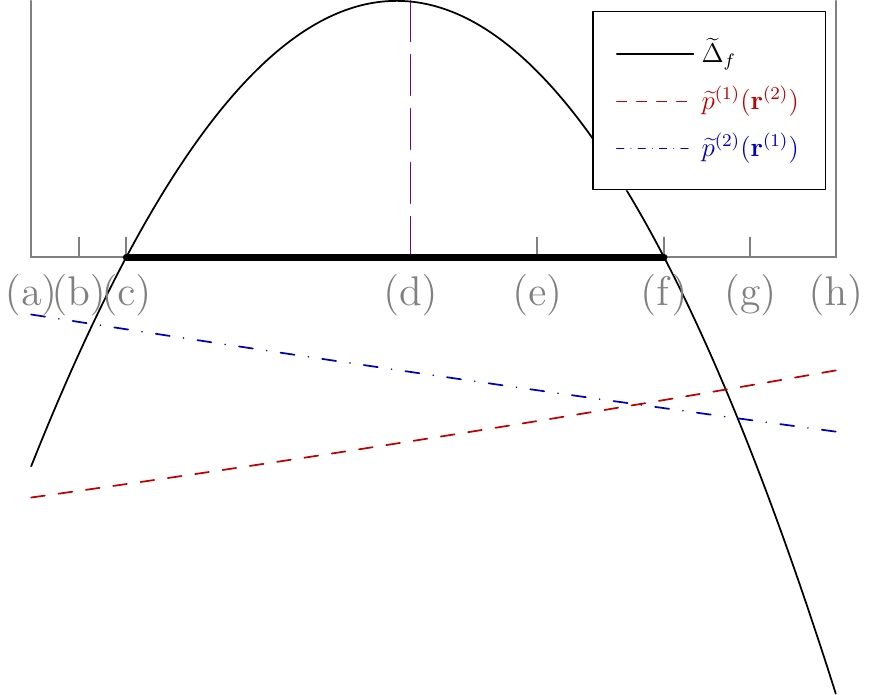}
    \caption{}
    \label{discriminant_tilde}
  \end{subfigure}  
  \caption{\emph{The rectangle $\F_{f,e}$ (outlined in black) and the sets $\tcE_f^{(1)}$ (red) and $\tcE_f^{(2)}$ (blue) at test volume fractions (a) $f=\widetilde{f}_{e,l}\approx 0.7982$; (b) $f \approx 0.7984$; (c) $f = \approx 0.7987$ (where $\widetilde{\Delta}_f = 0$);  (d) $f = f^{(1)} = 0.80$; (e) $f \approx 0.8006$; (f) $f\approx 0.8012$ (where $\widetilde{\Delta}_f = 0$); (g) $f \approx 0.8016$; (h) $f = \widetilde{f}_{e,u} \approx 0.8020$.  The red (blue) dashed ellipse is the boundary of $\cE^{(1)}_f$ ($\cE^{(2)}_f$).  (i) A plot of $\widetilde{\Delta}_f$ (black solid line), $\frac{1}{5}\widetilde{p}_f^{(1)}(\mathbf{r}^{(2)})$ (red dashed line), and $\frac{1}{5}\widetilde{p}_f^{(2)}(\mathbf{r}^{(1)})$ (blue dash-dotted line) for $f \in \tA_e = [\widetilde{f}_{e,l},\widetilde{f}_{e,u}]$.  The parameters used to create this figure are the same as those in Figure \ref{elementaryboundsfigure}.  So we obtain the bounds $0.7987 \le f^{(1)} \le 0.8012$, which are better than the improved elementary bounds from Section \ref{section_improved_elementary_bounds} and Figure \ref{parabolas}.}}  \label{tildeellipses_fig}
\end{figure}

To search for geometries for which these more sophisticated bounds are attained one could look for geometries such that for some choice of real vectors $\mathbf{c}^{(1)}, \mathbf{d}^{(1)}$ not both zero and $\mathbf{c}^{(2)}, \mathbf{d}^{(2)}$ not both zero
\begin{equation}\label{eq:h1h2zero}
  \begin{cases} \mathbf{h}^{(1)}(\x; \mathbf{c}^{(1)}, \mathbf{d}^{(1)}) \equiv 0 & \text{for } \x \in \text{phase } 1 \\
  			      \mathbf{h}^{(2)}(\x; \mathbf{c}^{(2)}, \mathbf{d}^{(2)}) \equiv 0 & \text{for } \x \in \text{phase } 2.
  \end{cases}
\end{equation}
In this case $\widetilde{p}_f^{(1)}$ and $\widetilde{p}_f^{(2)}$ will both be zero and $(x,y)$ must be at an intersection point of the boundary of the elliptic disk $\tcE_f^{(1)}$ and the boundary of the elliptic disk $\tcE_f^{(2)}$.  Conversely if $(x,y)$ is at such an intersection point then \eqref{eq:h1h2zero} must hold.  Additionally we require that the two ellipses only touch at one point and the meaning of this condition in terms of fields is not so clear.  Therefore \eqref{eq:h1h2zero} is a necessary, but not sufficient, condition for attainability of the bounds.  A similar remark applies to the attainability of the ``more sophisticated'' bounds derived in Section \ref{section_more_sophisticated_bounds}.

\subsection{Degenerate Cases}
In this section we briefly discuss the degenerate cases.
If $\vplusone$ or $\vminusone \equiv 0$ ($\vplustwo$ or $\vminustwo \equiv 0$), then $\widetilde{p}^{(1)}_{f,\max} \equiv 0$ ($\widetilde{p}^{(2)}_{f,\max} \equiv 0$) for all $f \in \A_e$ by \eqref{ptildemax}, so we are unable to derive a tighter lower (upper) elementary bound.  If $\vplusminusa = 0$ for $\alpha = 1,2$ we again have $\widetilde{A}_e = \A_e$.  In summary we construct the following table for the restricted elementary set of admissible volume fractions, $\tA_e$, assuming $\eta^{(\alpha)} \ne 0$ for $\alpha = 1,2$.  As the table shows, if $\vplusminusa = 0$ we have $\widetilde{\A}_e = \A_e$ (which is given in \eqref{Ae}).  One can apply the procedure discussed in the paragraphs preceding Theorem \ref{Rperptheorem} to try to improve these elementary bounds.
\begin{center}
{\renewcommand{\arraystretch}{1.5}
\renewcommand{\tabcolsep}{0.2cm}
\begin{tabular}[c]{|c|c|c|c|c|}
\hline
 \begin{tabular}{@{}c@{}} $\vplusminusone \not\equiv 0$ \\ and $\vplusminustwo \not\equiv 0 $\end{tabular} & \begin{tabular}{@{}c@{}} $\vplusone$ or $\vminusone \equiv 0$ \\ and $\vplusminustwo \not\equiv 0 $\end{tabular} & \begin{tabular}{@{}c@{}} $\vplustwo$ or $\vminustwo \equiv 0$ \\ and $\vplusminusone \not\equiv 0 $\end{tabular}  & \begin{tabular}{@{}c@{}} $\vplusminusone \equiv 0$ \\ and $\vplusminustwo \equiv 0 $\end{tabular} \\
\hline
$[\widetilde{f}_{e,l},\widetilde{f}_{e,u}]$ & $[f_{e,l}, \widetilde{f}_{e,u}]$ & $[\widetilde{f}_{e,l}, f_{e,u}]$ & $[f_{e,l}, f_{e,u}] = \A_e$\\
\hline
\end{tabular}}
\end{center}


\section{Numerical Example}\label{section_numerical_example}

In this section we present the results of several numerical experiments.  We used the two dimensional configuration and boundary conditions from Figure \ref{elementaryboundsfigure} to create the plots in Figure \ref{numerical_figure}.  In each subplot $\sigma^{(1)}$ is fixed and $\sigma^{(2)} = 1$; we varied the volume fraction by fixing $R_1 = 0.45$ and $R_3 = 5$ while varying $R_2$ between approximately $0.6727$ and $4.995$.  

Each subplot contains the following data scaled by $f^{(1)}$: $f_{e,l}$ (red stars); $\inf \A$ (red circles); $\widetilde{f}_{e,l}$ (red crosses); $\inf \tA$ (red squares); $f_{e,u}$ (blue stars); $\sup \A$ (blue circles); $\widetilde{f}_{e,u}$ (blue crosses); $\sup \tA$ (blue squares).  In all of the plots, $f_{e,l}/f^{(1)} = \inf \A/f^{(1)}$ and $f_{e,u}/f^{(1)} = \sup \A/f^{(1)}$, so the bounds obtained by using the elliptic disks $\cE_f^{(1)}$ and $\cE_f^{(2)}$ from Section \ref{section_more_sophisticated_bounds} (namely $\inf \A$ and $\sup \A$) are simply the elementary bounds $f_{e,l}$ and $f_{e,u}$ from Section \ref{section_elementary_bounds}.  

For many cases in this 2-D example the bounds obtained by using the elliptic disks $\tcE^{(1)}_f$ and $\tcE^{(2)}_f$ from Section \ref{section_tilde_ellipse_bounds} (namely $\inf \tA$ and $\sup \tA$) are substantially better than the improved elementary bounds $\widetilde{f}_{e,l}$ and $\widetilde{f}_{e,u}$ from Section \ref{section_improved_elementary_bounds}. In particular, the extra information from the elliptic disks $\tcE^{(1)}$ and $\tcE^{(2)}$ gives us lower bounds that, most of the time, are better than the improved elementary bounds $\widetilde{f}_{e,l}$ and $\widetilde{f}_{e,u}$; this extra information does not seem to improve the upper bound in most cases, however.

\begin{figure}[!hbtp]
  \centering
  \begin{subfigure}{0.45\textwidth}
    \centering
    \includegraphics[scale=0.8]{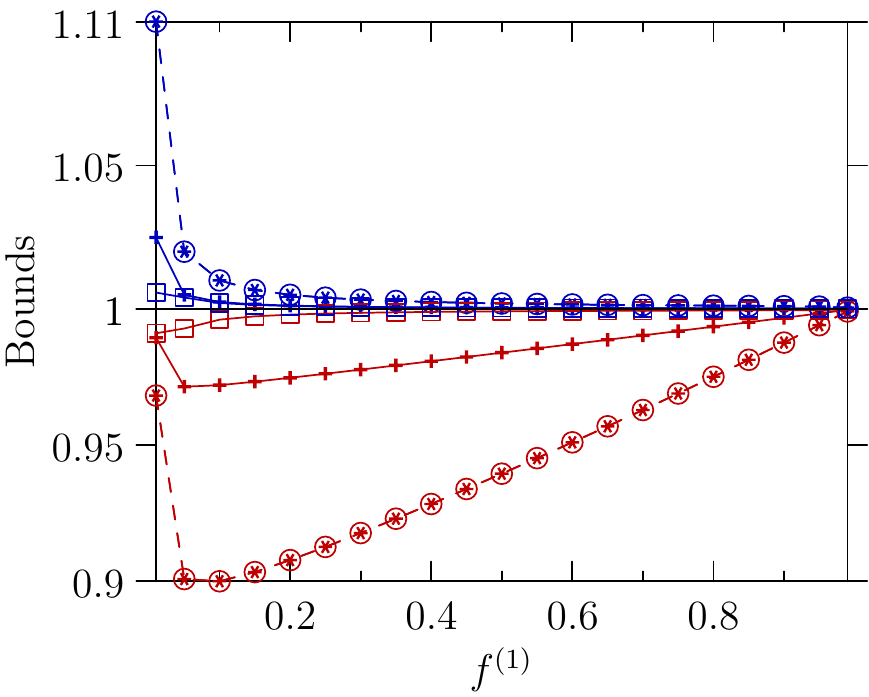}
    \caption{}
    \label{twopluspointfivei}
  \end{subfigure}
  \quad
  \begin{subfigure}{0.45\textwidth}
    \centering
    \includegraphics[scale=0.8]{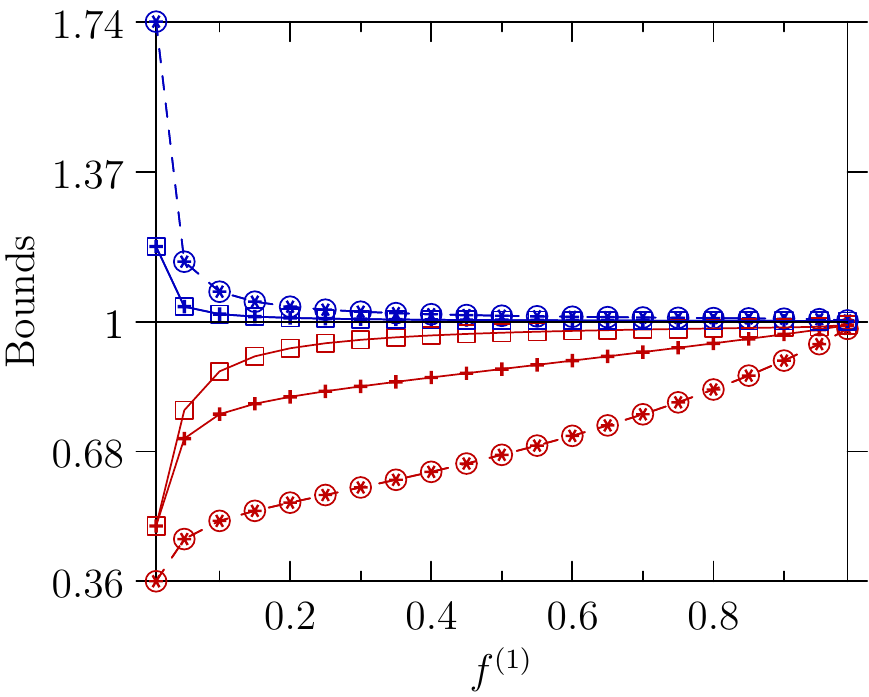}
    \caption{}
    \label{twoplusteni}
  \end{subfigure}
  \\[0.5cm]
  \begin{subfigure}{0.45\textwidth}
    \centering
    \includegraphics[scale=0.8]{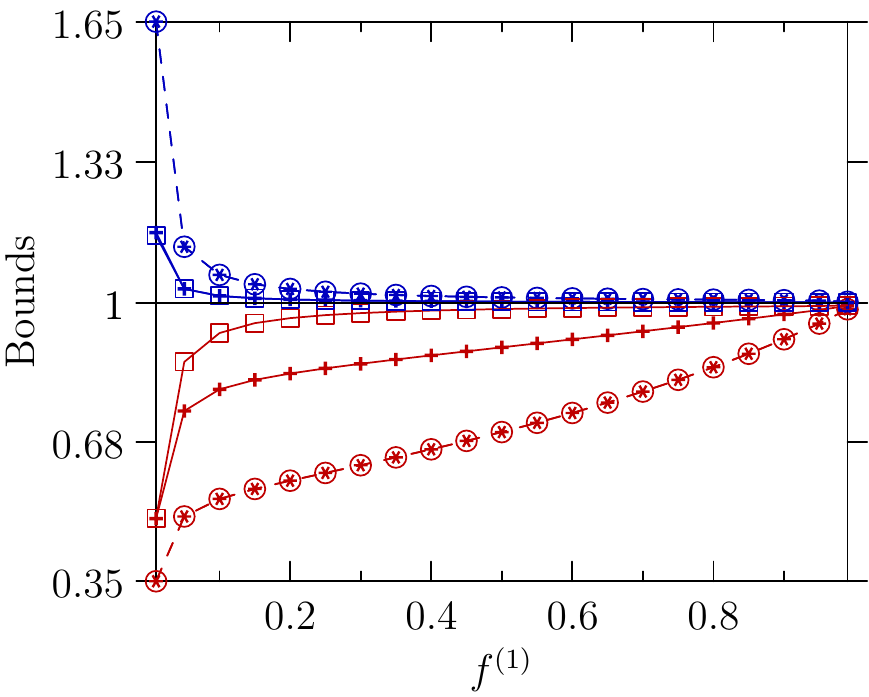}
    \caption{}
    \label{tenplusteni}
  \end{subfigure}
  \quad
  \begin{subfigure}{0.45\textwidth}
    \centering
    \includegraphics[scale=0.8]{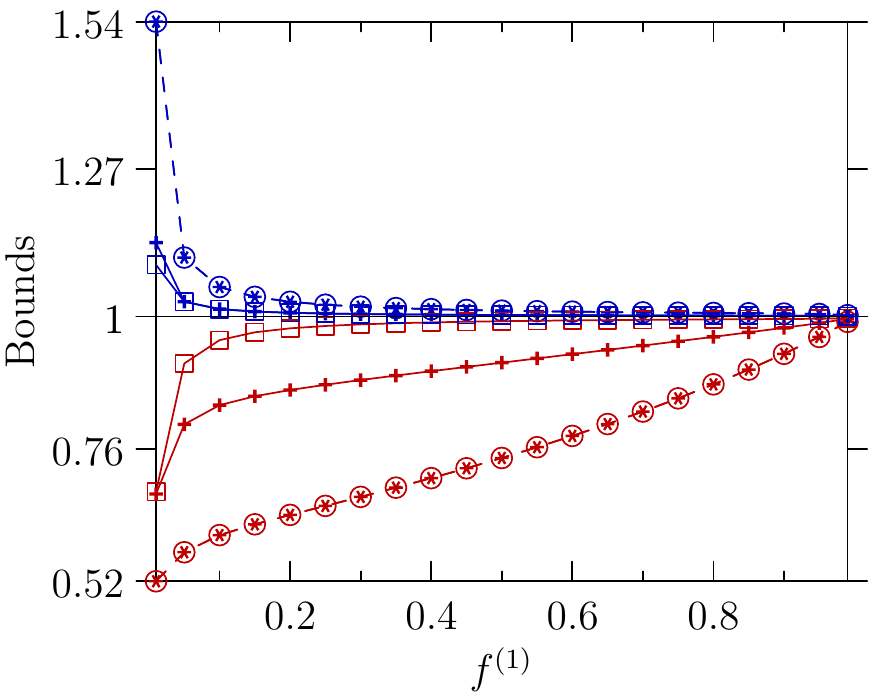}
    \caption{}
    \label{tenpluspointfivei}
  \end{subfigure}
  \\[0.5cm]
  \begin{subfigure}{\textwidth}
    \centering
    \includegraphics[scale=1.25]{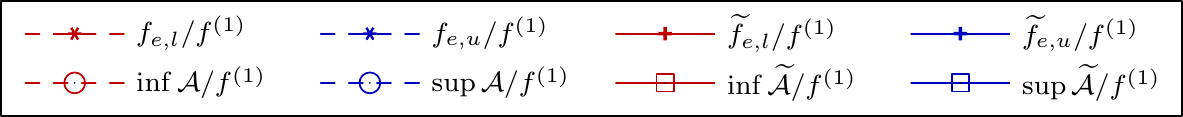}
  \end{subfigure}
  \caption{\emph{A plot of the bounds in the case of an annulus (see Figure \ref{annulus_fig}) for several volume fractions ranging from $f^{(1)} = 0.01$ to $f^{(1)} = 0.99$.  In each subfigure the conductivity of the annular ring is $\sigma^{(2)} = 1$ while the conductivity of the surrounding medium in each subfigure is: (a) $\sigma^{(1)} = 2+0.5\ii$; (b) $\sigma^{(1)} = 2+10\ii$; (c) $\sigma^{(1)} = 10+10\ii$; (d) $\sigma^{(1)} = 10+0.5\ii$.  The legend at the bottom indicates the symbol used to represent each bound; in particular we used the following labels: red circles--elementary lower bound ($f_{e,l}$--see Section \ref{section_elementary_bounds}); red stars--``sophisticated'' lower bound (see  Section \ref{section_more_sophisticated_bounds}); red crosses--improved elementary lower bound ($\widetilde{f}_{e,l}$--see Section \ref{section_improved_elementary_bounds}); red squares--improved ``sophisticated'' lower bound (see Section \ref{section_tilde_ellipse_bounds}); blue circles--elementary upper bound ($f_{e,u}$--see Section \ref{section_elementary_bounds}); blue stars--``sophisticated'' upper bound (see  Section \ref{section_more_sophisticated_bounds}); blue crosses--improved elementary upper bound ($\widetilde{f}_{e,u}$--see Section \ref{section_improved_elementary_bounds}); blue squares--improved ``sophisticated'' upper bound (see Section \ref{section_tilde_ellipse_bounds}).}}  \label{numerical_figure}
\end{figure}

\section*{Acknowledgements} Graeme Milton especially wishes to thank George Papanicolaou for help and encouragement and insightful suggestions at many times during his career and for inspiring him to write the book \emph{The Theory of Composites}.  Both authors are thankful for support from the National Science Foundation through grants DMS-0707978 and DMS-1211359 and are grateful to Hyeonbae Kang for stimulating conversations.


\newpage

\bibliographystyle{plain}
\bibliography{volume_fraction_bounds_most_recent}

\end{document}